\documentclass[11pt]{article}

\usepackage[T1]{fontenc}
\usepackage{amsmath,amssymb,amsthm}
\usepackage[shortlabels]{enumitem}
\usepackage[margin=1in]{geometry}
\usepackage[hidelinks]{hyperref}

\DeclareMathOperator{\diam}{diam}
\DeclareMathOperator{\Div}{div}
\DeclareMathOperator{\supp}{supp}
\newcommand{\R}{\mathbb{R}}

\numberwithin{equation}{section}
\allowdisplaybreaks
\theoremstyle{plain}
\newtheorem{theorem}{Theorem}[section]
\newtheorem{lemma}[theorem]{Lemma}
\newtheorem{corollary}[theorem]{Corollary}
\newtheorem{proposition}[theorem]{Proposition}

\theoremstyle{definition}
\newtheorem{definition}[theorem]{Definition}
\newtheorem{remark}[theorem]{Remark}

\begin{document}

\title{Stability of Serrin-type overdetermined problems for the
  $p$-Laplacian and related operators}

\author{%
  Yingfang Zhang\thanks{Corresponding author. Email:
    \href{mailto:zhangyf22@mails.tsinghua.edu.cn}%
    {zhangyf22@mails.tsinghua.edu.cn}.}
  \quad
  Wenming Zou\thanks{Email:
    \href{mailto:zou-wm@mail.tsinghua.edu.cn}%
    {zou-wm@mail.tsinghua.edu.cn}.}\\[0.5em]
  \small Department of Mathematical Sciences, Tsinghua University\\
  \small Beijing 100084, China
}
\date{}

\maketitle

\begin{abstract}
We investigate a Serrin-type overdetermined problem for the $p$-Laplacian and establish a quantitative stability estimate under suitable convexity and regularity assumptions on the domain, the solution, and the operator. Our method also applies to a class of generalized divergence-form elliptic operators.
\end{abstract}



\section{Introduction}
Overdetermined elliptic problems provide a classical mechanism through which information about a solution determines the geometry of the domain. A typical example is Serrin's problem: if a bounded domain $\Omega\subset \R^N$ admits a solution of
\begin{equation*}
    \begin{cases}
        -\Delta u = 1 & \text{in $\Omega$},\\
        u=0,\quad |\nabla u|=\mathbf{c} & \text{on $\partial\Omega$}
    \end{cases},
\end{equation*}
then $\Omega$ is a ball. Here $N\ge 2$ and $\mathbf{c}>0$ is a constant. Serrin \cite{Serrin1971} introduced the moving-plane method and proved this result under the assumptions that $u\in C^2(\bar{\Omega})$ and $\Omega$ is of class $C^2$. Shortly afterward, Weinberger \cite{Weinberger1971} gave a different proof based on a ``P-function'' and an integral identity. These arguments initiated two of the principal approaches to overdetermined problems.

Serrin-type problems have been extended to nonlinear and quasilinear operators. For the $p$-Laplacian equation, Garofalo and Lewis \cite{Garofalo1989} proved a symmetry theorem for a broad class of quasilinear overdetermined problems. Damascelli and Pacella \cite{Damascelli2000} obtained the corresponding Serrin consequence for $1<p<2$ through the moving plane method, while Brock and Henrot \cite{Brock2002a} developed a proof based on continuous rearrangement and domain derivatives. Continuous Steiner symmetrization has subsequently become a useful tool for symmetry questions involving
degenerate elliptic equations. In particular, Cao, Wei, and Zhan \cite{Cao2025} established Serrin-type symmetry results for a class of divergence-form equations in nonsmooth domains.

Other extensions include the result of Kumaresan and Prajapat \cite{Kumaresan1998} in hyperbolic space, the $p$-Laplacian result of Ruan, Huang, and Chen \cite{Ruan2024} on Riemannian manifolds, and the rough-domain theorem of Figalli and Zhang \cite{Figalli2025} for the classical Serrin problem.


A natural quantitative question is whether an almost constant Neumann datum forces the domain to be close to a ball. The first quantitative result for Serrin's problem using moving-plane is due to Aftalion, Busca, and Reichel \cite{Aftalion1999}, who obtained an approximate radial-symmetry estimate of logarithmic type. For the torsion problem, Brandolini, Nitsch, Salani, and Trombetti \cite{Brandolini2008a} obtained stronger quantitative estimates using integral identities. Ciraolo, Magnanini, and Vespri \cite{Ciraolo2016} later established H\"{o}lder stability for a broad class of semilinear equations using a refined quantitative moving-plane method. Another line of work uses integral identities and geometric inequalities: Feldman \cite{Feldman2018} obtained a linear stability estimate controlled by the weak boundary quantity $\lVert |\nabla u|^2-\mathbf{c}^2\rVert_{L^2(\partial\Omega)}$, while Magnanini and Poggesi \cite{Magnanini2020} developed improved, nearly optimal estimates for Serrin's problem and clarified its close connection with Alexandrov's soap-bubble theorem. Quantitative stability has also been extended to nonlocal Serrin problems; see \cite{Dipierro2023, Thompson2025}.

Related quantitative symmetry results for the $p$-Laplacian were recently obtained by Ciraolo and Li \cite{Ciraolo2025a}, who estimated the deviation from Schwarz symmetry in terms of the isoperimetric deficit, and by Dipierro, Gonçalves da Silva, Poggesi, and Valdinoci \cite{DipierroPog2025}, who developed a quantitative Gidas--Ni--Nirenberg-type result via integral identities. These works concern quantitative symmetry of solutions, whereas here we derive quantitative closeness of the domain to a ball from an almost constant boundary gradient.

The present work follows a different quantitative route, motivated by the continuous Steiner symmetrization proof of qualitative symmetry. For $1<p<\infty$, we consider the $p$-Laplacian problem 
\begin{equation}\label{Eq1}
    \begin{cases}
        -\Delta_p u = f(u) & \text{in $\Omega$}, \\
        u = 0& \text{on $\partial\Omega$},
    \end{cases}
\end{equation}
with boundary defect $\epsilon_0=\lVert |\nabla u|-\mathbf{c}\rVert_{L^{\infty}(\partial\Omega)}$. Our main result is as follows.

\begin{theorem}\label{mainthm}
    Assume that, for some $s_*>0$, $L>0$, and $\alpha\in(0,1]$, the following conditions hold:
    \begin{enumerate}[(a)]
        \item $\Omega$ is a bounded domain in $\R^N$, $f$ is continuous and nonnegative, and $u\in C^1(\bar\Omega)$ solves \eqref{Eq1} with $\big||\nabla u|-\mathbf{c}\big|\le \epsilon_0$ on $\partial\Omega$, where $\epsilon_0 < \min\{1,\mathbf{c}/4\}$;
        \item $|\nabla u(x)-\nabla u(y)| \le L|x-y|^\alpha$ in $\bar{\Omega}$;
        \item every superlevel set $\{u>r\}$ with $0\le r\le s_*$ is convex;
        \item $\partial\Omega$ is of class $C^1$, and the inradius $\rho_i$ of $\Omega$ is positive.
    \end{enumerate}
    Then there exists a constant $C>0$ depending only on $\diam(\Omega), \rho_i,N,p,\mathbf{c}$, $C_0:=\lVert u\rVert_{C^1(\bar\Omega)}$, $C_1:=\lVert f\rVert_{C^0([0,C_0])}$, $L$, $\alpha$, and $s_*$ such that
    \begin{equation}\label{main est}
        |\Omega\Delta B_\Omega(q)| \le C \epsilon_0^{\frac{\alpha}{2(\alpha+1)}}
    \end{equation}
    for some $q\in\R^N$, where $B_{\Omega}(q)$ denotes the ball centered at $q$ with the same volume as $\Omega$.
\end{theorem}

Under additional regularity and convexity assumptions on the data and the domain, we obtain the following corollary.

\begin{corollary}\label{maincor1}
    Suppose $\Omega$ is a bounded domain in $\R^N$ and that \eqref{Eq1} admits a solution $u\in C^{1}(\bar\Omega)$ with $\epsilon_0<\min\{1,\mathbf{c}/4\}$ and $C_0:=\lVert u\rVert_{C^1(\bar \Omega)}<\infty$.
    \begin{enumerate}[(i)]
        \item If $\Omega$ is uniformly convex, $\partial\Omega$ is of class $C^{2,\beta}$, and $f\in C^\beta([0,C_0])$ is nonnegative for some $\beta\in(0,1]$, then there exists a constant $C''>0$ depending only on $N$, $p$, $\beta$, $\mathbf{c}$, $C_0$, $\lVert f\rVert_{C^\beta([0,C_0])}$, $\diam(\Omega)$, the lower curvature bound $\kappa_0$ of $\partial\Omega$, the inradius $\rho_i$ of $\Omega$ and $\mathcal{C}_{2,\beta}(\Omega)$ such that
        \begin{equation}\label{main}
            |\Omega\Delta B_{\Omega}(q)| \le C''\epsilon_0^{1/4}
        \end{equation}
        for some $q\in\R^N$, where $\mathcal{C}_{2,\beta}(\Omega)$ is defined as in Definition \ref{character}.

        \item If $\Omega$ is convex, $\partial\Omega$ is of class $C^2$, $f\in C^\beta([0,C_0])$ for some $\beta\in(0,1]$, and $f$ is H\"{o}lder continuous, nonnegative on $[0,\infty)$. Moreover,
        \[ x\mapsto \frac{f(x)}{x^{p-1}} \text{ is nonincreasing},\quad x\mapsto \frac{e^{(p-1)x}}{f(e^x)} \text{ is convex on $(-\infty,\log M)$},\]
        where $M=\inf\{t>0:f(t)=0\}$. Then there exist $\alpha\in(0,1]$ and $C''>0$, depending only on $N$, $p$, $\beta$, $\mathbf{c}$, $C_0$, $\lVert f\rVert_{C^\beta([0,C_0])}$, $\diam(\Omega)$, the inradius $\rho_i$ of $\Omega$ and $\mathcal{C}_2(\Omega)$, such that
        \begin{equation}\label{main est 2}
            |\Omega\Delta B_\Omega(q)| \le C'' \epsilon_0^{\frac{\alpha}{2(\alpha+1)}}
        \end{equation}
        for some $q\in\R^N$, where $\mathcal{C}_{2}(\Omega)$ is defined as in Definition \ref{character}.
        
        If, in addition, $\partial\Omega$ is of class $C^{2,\beta}$, then one may take $\alpha=1$ in \eqref{main est 2}, and the constant $C''$ additionally depends on $\mathcal{C}_{2,\beta}(\Omega)$.
    \end{enumerate}
\end{corollary}

The argument also extends from $-\Delta_p u$ to the operator $-\Div \big(g(|\nabla u|)\frac{\nabla u}{|\nabla u|}\big)$, where $g:[0,\infty)\to \R$ is strictly increasing, and yields a similar stability estimate under suitable structural assumptions.

\begin{corollary}\label{Cor2}
    Let $g\in C([0,\infty)) \cap C^1((0,\infty))$, $g(0)=0$, $g(x)>0$ for $x>0$ and
    \begin{align}
        0< \lambda_g \le \frac{g'(x)x}{g(x)} \le \Lambda_g <\infty,&\quad \forall\, x>0;  \label{eigen} \\
        x\mapsto \frac{xg'(x)}{g(x)} \in C_{\text{loc}}^{0,\theta}((0,\infty)),&\quad \text{for some $\theta\in(0,1]$}. \label{regular g}
    \end{align}
    Suppose $\Omega$ is uniformly convex, $\partial\Omega$ is of class $C^{2,\beta}$, and $f\in C^\beta(\R)$ is nonnegative. If there exists $u\in C^1(\bar\Omega)$ such that
    \begin{equation}\label{g ver eq}
        \begin{cases}
            -\Div \big(g(|\nabla u|)\frac{\nabla u}{|\nabla u|}\big) = f(u) & \text{in $\Omega$},\\
            u = 0 & \text{on $\partial\Omega$},
        \end{cases}
    \end{equation}
    and $\epsilon_0 :=\lVert |\nabla u|-\mathbf{c}\rVert_{L^\infty(\partial\Omega)} < \min\{1,\mathbf{c}/4\}$, then there exists a constant $C_g$ depending only on $N$, $\mathbf{c}$, $\theta$, $\beta$, $g$, $\kappa_0$, $\lVert f\rVert_{C^\beta(\R)}$, $\lVert u\rVert_{C^1(\bar\Omega)}$, $\diam(\Omega)$, $\rho_i$ and $\mathcal{C}_{2,\beta}(\Omega)$ such that
    \begin{equation}
        |\Omega \Delta B_\Omega(q)| \le C_g \epsilon_0^{1/4}.
    \end{equation}
    Here $\mathcal{C}_{2,\beta}(\Omega)$ is defined as in Definition \ref{character}.
\end{corollary}

To the best of our knowledge, and within the class of estimates controlled directly by the boundary oscillation of $|\nabla u|$, the present result provides the first quantitative stability estimate for the $p$-Laplacian Serrin problem in terms of the deviation of the boundary gradient from a constant.

Our main tool is Steiner symmetrization. Inspired by the continuous Steiner symmetrization method in \cite{Cao2025, Brock2000}, we first convert the boundary defect $\epsilon_0$ into a quantitative deficit in the $p$-Dirichlet energy. After truncating the solution near the boundary, following an idea from \cite{Fusco2009}, we compare the energy deficit with the decrease in the perimeters of the superlevel sets. We then apply the quantitative stability of Steiner symmetrization for convex sets to obtain approximate reflection symmetry in every direction. A final geometric argument converts these directional estimates into a bound for the Fraenkel asymmetry of $\Omega$. The same mechanism applies to more general divergence-form operators under suitable regularity and structural assumptions.

The paper is organized as follows. Section \ref{sec 2} collects a maximum principle and related properties of the solution, and reviews Steiner symmetrization. In Section \ref{sec3}, we prove Proposition \ref{Prop 1} and use it to establish Theorem \ref{mainthm}. Section \ref{sec4} proves Corollary \ref{maincor1} and extends the method to a class of divergence-form elliptic operators.

\section{Preliminaries}\label{sec 2}

Throughout the paper, $\nu(x)$ denotes the inward unit normal to $\partial\Omega$ at $x\in\partial\Omega$.
We denote by $w_N$ the volume of the unit ball in $\R^N$.

\subsection{Properties of the solution}
\begin{lemma}\label{MP}
    Suppose $D$ is a bounded domain in $\R^N$. If $v$ satisfies $-\Delta_p v \le 0$ in $D$, and $v\le M$ on $\partial D$ for some $M\ge 0$, then $v\le M$ in $D$.
\end{lemma}
This follows directly from Pucci and Serrin \cite[Theorem 3.2.1]{Pucci2007}.

For $\delta>0$, denote by $\Omega_{\delta}=\{x\in\Omega: d(x,\partial\Omega)\ge \delta\}$.

\begin{corollary}\label{Cor1}
    Suppose $u$ is a solution of \eqref{Eq1} and satisfies conditions (a), (b), and (d) in Theorem \ref{mainthm} with $\epsilon_0<\mathbf{c}/2$. Set $C_0:=\lVert u\rVert_{C^1(\bar\Omega)}$ and $C_1:=\lVert f\rVert_{C^0([0,C_0])}$. Then for any $\delta < \bigl(\frac{\mathbf{c}/2-\epsilon_0}{L}\bigr)^{1/\alpha}$, we have
    \begin{enumerate}[(1)]
        \item $|\nabla u|>\mathbf{c}/2$ for all $x\in \Omega\backslash \Omega_{\delta}$.
        \item For $x\in \Omega\backslash \Omega_{\delta}$, $u(x) \ge \frac{\mathbf{c}}{2}d(x,\partial\Omega)$ and $u\ge \delta \mathbf{c}/2$ in $\Omega_{\delta}$.
        \item There exists a constant $C_2=C_1|\Omega|>0$ such that
        \begin{equation}
            \int_{\{u=\gamma\}} |\nabla u|^{p-1}\,d\mathcal{H}^{N-1}(x) \le C_2,
        \end{equation}
        for all $\gamma\in (0,\delta\mathbf{c}/2)$. Here $\mathcal{H}^{N-1}$ denotes the $(N-1)$-dimensional Hausdorff measure.
    \end{enumerate}
\end{corollary}
\begin{proof}
    Since $u=0$ on $\partial\Omega$ and $-\Delta_p u=f(u)\ge 0$ in $\Omega$, Lemma \ref{MP} and the strong maximum principle imply that $u>0$ in $\Omega$.

    For $x\in\Omega$, choose $y\in\partial\Omega$ such that $|x-y|=d(x,\partial\Omega)$. Then $\frac{x-y}{|x-y|} = \nu(y)$. By condition (b),
    \[ |\nabla u(x)| \ge |\nabla u(y)| - |\nabla u(x) - \nabla u(y)| \ge \mathbf{c}- \epsilon_0 - L|x-y|^\alpha \ge \mathbf{c}/2 \]
    provided $d(x,\partial\Omega) \le \bigl(\frac{\mathbf{c}/2-\epsilon_0}{L}\bigr)^{1/\alpha}$, and this proves (1).
    
    By the mean value theorem, there is $\theta\in(0,1)$ such that
    \[ u(x) = u(y) + \nabla u(z) \cdot (x-y),\quad z=\theta x+(1-\theta)y.\]
    Again by condition (b), 
    \[ \Big|\nabla u(z)\cdot \frac{x-y}{|x-y|} - \nabla u(y)\cdot \nu(y)\Big| \le |\nabla u(z)-\nabla u(y)| \le L|z-y|^\alpha \le L|x-y|^\alpha.\]
    Since $u=0$ on $\partial\Omega$ and $u>0$ in $\Omega$, we know $\nabla u(y)\cdot\nu(y) =|\nabla u(y)| \ge \mathbf{c}-\epsilon_0$, and so
    \begin{align*}
        u(x) = \nabla u(z)\cdot (x-y) \ge{}& |x-y|\nabla u(y)\cdot\nu(y) - L|x-y|^{\alpha+1}\\
        \ge{}& (\mathbf{c}-\epsilon_0 - L|x-y|^{\alpha})|x-y|\\
        ={}& (\mathbf{c}-\epsilon_0 - Ld(x,\partial\Omega)^{\alpha})d(x,\partial\Omega).
    \end{align*}
    Hence for $x\in\Omega\backslash \Omega_{\delta}$, we have $u(x) \ge \frac{\mathbf{c}}2 d(x,\partial\Omega)$. Let $v=\delta\mathbf{c}/2-u$, then
    \[ \begin{cases}
        -\Delta_p v = -f(u)\le 0 & \text{in $\Omega_{\delta}$},\\
        v\le 0 & \text{on $\partial\Omega_{\delta}$}.
    \end{cases}\] 
    By Lemma \ref{MP}, $v\le 0$ in $\Omega_{\delta}$; hence $u\ge \delta\mathbf{c}/2$ in $\Omega_{\delta}$. This proves (2).

    Since $|\nabla u|$ does not vanish in $\{0<u<\delta\mathbf{c}/2\}$, by the implicit function theorem, $\{u>\gamma\}$ is an open set of $\Omega$ with $\partial\{u>\gamma\}=\{u=\gamma\}$ for every $\gamma\in(0,\delta\mathbf{c}/2)$, and $\{u=\gamma\}$ is locally a $C^1$-hypersurface. Integrating \eqref{Eq1} over $\{u>\gamma\}$ and using the divergence theorem, we have
    \[ \int_{\{u=\gamma\}} |\nabla u|^{p-1}\,d\mathcal{H}^{N-1}(x) = \int_{\{u>\gamma\}} f(u)\,dx \le C_1|\Omega|.\]
    This proves (3).
 \end{proof}

\subsection{Steiner symmetrization}
After a translation, we may assume throughout the following sections that $0\in\Omega$; consequently, $\Omega\subset B_{\diam(\Omega)}(0)$.

We introduce the continuous Steiner symmetrization for completeness. All the properties can be found in Brock \cite[Section 2]{Brock2000}. By $\mathcal{M}(\R^N)$, we denote the family of Lebesgue measurable sets in $\R^N$ with finite measure. For a function $u:\R^N\to \R$, let $\{u>a\}$ denote the set $\{x\in\R^N: u(x)>a\}$ ($a\in \R$). Let $\mathcal{S}(\R^N)$ be the set of real-valued, nonnegative measurable functions $u$ that satisfies $|\{u>c\}|<\infty$ for all $c>0$.

\begin{definition}[Continuous Steiner symmetrization of sets in $\mathcal{M}(\R)$]
    Assume
    \[ \mathcal{T}_t: \mathcal{M}(\R)\to \mathcal{M}(\R), \quad 0\le t\le \infty\]
    is a family of set transformations. We call it a continuous symmetrization on $\R$ if it satisfies the following properties:
    \begin{enumerate}
        \item equimeasurability: for $M\in \mathcal{M}(\R)$, $|\mathcal{T}_t(M)| = |M|$,
        \item monotonicity: if $M\subset E$ are sets in $\mathcal{M}(\R)$, then $\mathcal{T}_t(M)\subset \mathcal{T}_t(E)$,
        \item semigroup property: $\mathcal{T}_t (\mathcal{T}_s (M)) = \mathcal{T}_{s+t}(M)$,
        \item interval evolution: if $M=[x-R,x+R]$, then $\mathcal{T}_t(M)=[xe^{-t}-R,xe^{-t}+R]$.
    \end{enumerate}
\end{definition}
It can be verified that $M^0 = M$ and $M^\infty$ is the Steiner symmetrization of $M$. Moreover, $M^t$ is open (resp. closed) if $M$ is open (resp. closed).

For sets in $\R^N$, we can define continuous Steiner symmetrization with respect to a given direction $\tau\in\R^N$. In the following, we assume $\tau = e_N$.

\begin{definition}[Continuous Steiner symmetrization of sets in $\mathcal{M}(\R^N)$]
    Let $M\in \mathcal{M}(\R^N)$, $N\ge 2$. Then the family of sets
    \[ M^t := \{x=(x',y): y\in (M(x'))^t, \, x'\in\R^{N-1}\},\quad t\in[0,\infty]\]
    is called the continuous Steiner symmetrization of $M$, where $M(x')=\{y\in\R: (x',y)\in M\}$.
\end{definition}

\begin{definition}[Continuous Steiner symmetrization of functions]
    Let $u\in \mathcal{S}(\R^N)$ ($N\ge 2$). Then
    \[ u^t(x) := \begin{cases}
         \sup\{ c>\inf u: x\in \{u>c\}^t \}, & \text{if $x\in \bigcup\limits_{c>\inf u} \{u>c\}^t$ } \\
         \inf u, & \text{otherwise}
    \end{cases},\quad x\in\R^N\]
    is called the continuous Steiner symmetrization of $u$ with respect to $e_N$.
\end{definition}

We state some properties that are needed in the subsequent proofs.

\begin{proposition}\label{Steiner prop}
    Let $M\in\mathcal{M}(\R^N)$, $u,v\in\mathcal{S}(\R^N)$, and $t\in[0,\infty]$. Then
    \begin{enumerate}[(1)]
        \item equimeasurability: $|M| = |M^t|$ and $\{u^t>c\} = \{u>c\}^t$, $\forall\, c>0$.
        \item monotonicity: if $u\le v$, then $u^t\le v^t$.
        \item Cavalieri's principle: if $F$ is Borel measurable and $F(u)\in L^1(\R^N)$, then
        \[ \int_{\R^N} F(u) = \int_{\R^N} F(u^t).\]
        \item continuity with respect to the parameter $t$: if $t_m\to t$ as $m\to \infty$, then $M^{t_m}\to M^t$ in measure. If in addition $u$ is a.e. finite, then $u^{t_m}\to u^t$ in measure.
        \item If $u$ is Lipschitz continuous with Lipschitz constant $L$, then $u^t$ is Lipschitz continuous with Lipschitz constant less than $L$.
        \item If $\supp u \subset  B_R$ for some $R>0$, then $\supp u^t\subset B_R$ as well. If in addition $u$ is Lipschitz continuous with Lipschitz constant $L$, then
        \[ |u(x)-u^t(x)| \le LRt,\quad \forall\, x\in B_R.\]
        Furthermore,
        \[ \int_{B_R} G(|\nabla u^t|)\,dx \le \int_{B_R} G(|\nabla u|)\,dx\]
        for every convex function $G:[0,\infty)\to [0,\infty)$ with $G(0)=0$.
    \end{enumerate}
\end{proposition}

\begin{lemma}\label{commute}
    For $u\in\mathcal{S}(\R^N)$ and $s\in\R$, let $u_s := \min\{u,s\}$ and $(u-s)_+ := \max\{u-s,0\}$. Then
    \[ (u_s)^t=(u^t)_s,\quad (u-s)_+^t = (u^t-s)_+\]
    for all $t\in[0,\infty]$.

    Consequently, for $0<\gamma<s$, define $u_{\gamma,s}:=\min\{(u-\gamma)_+,s-\gamma\} = (u-\gamma)_+-(u-s)_+$. Then $(u_{\gamma,s})^t = (u^t)_{\gamma,s}$.
\end{lemma}
\begin{proof}
    For $0<c\le s$, it is easy to verify $\{u>c\} = \{u_s>c\}$, and so $\{u>c\}^t=\{u_s>c\}^t$. Thus,
    \begin{align*}
        (u_s)^t(x) ={}& \sup\{0<c\le s: x\in\{u_s>c\}^t\}\\
        ={}& \sup\{0<c\le s: x\in\{u>c\}^t\} \\
        ={}& \min\{\sup\{c>0: x\in \{u>c\}^t\},s\} \\
        ={}& (u^t)_s(x), \quad \forall\, x\in\Omega.
    \end{align*}
    Similarly, for $c>s$, $\{u>c\}^t = \{(u-s)_+>c-s\}^t$ and so
    \begin{align*}
        (u-s)_+^t(x) ={}& \sup\{c>0: x\in \{(u-s)_+>c\}^t\} \\
        ={}& \sup\{c>0: x\in \{u>c+s\}^t\} \\
        ={}& \begin{cases}
            0, & \text{if $u^t(x)\le s$}\\
            u^t(x)-s, & \text{if $u^t(x)> s$}
        \end{cases} \\
        ={}& (u^t-s)_+,\quad \forall\, x\in\Omega.
    \end{align*}
    Hence, for $0<\gamma<s$,
    \[ (u_{\gamma,s})^t = \min\{(u-\gamma)_+^t,s-\gamma\} = \min\{(u^t-\gamma)_+, s-\gamma\} = (u^t)_{\gamma,s}. \qedhere\]
\end{proof}

\section{Quantitative Stability via Steiner Symmetrization}\label{sec3}

We first show that $\epsilon_0$ controls the deficit in the $p$-Dirichlet energy. We begin with the required convexity property.

\begin{lemma}\label{good convex}
    Let $v\in W_0^{1,\infty}(\Omega)$ with $v\ge 0$, $v\in C^1\{0<v<\lVert v\rVert_{L^\infty(\Omega)}\}$, and assume that $\{v>v_0\}$ is convex for every $0\le v_0<\lVert v\rVert_{L^{\infty}(\Omega)}$. For $\tau\in \mathbb{S}^{N-1}$, let $v^{t,\tau}$ denote the continuous Steiner symmetrization of $v$ in the direction $\tau$, and define
    \[ A_v = \{\tau\in \mathbb{S}^{N-1} : |\{x\in\R^N: \nabla v(x)\neq 0,\, \nabla v(x)\cdot\tau =0\}| = 0\}.\]
    Then $|A_v|=|\mathbb{S}^{N-1}|$. Moreover, for every $\tau\in A_v$,
    \[ \phi(\lambda):=\int_{\R^N} |\nabla v^t|^p\]
    is decreasing and convex with respect to $\lambda= \frac{1-e^{-t}}{2}\in [0,1/2]$.
\end{lemma}
\begin{proof}
    We first show that $|A_v|=|\mathbb{S}^{N-1}|$. For a fixed $x\in \R^N$ such that $\nabla v(x)\neq 0$,
    \[ \{\tau\in\mathbb{S}^{N-1}: \nabla v(x)\cdot\tau = 0\}\]
    is a null set of $\mathbb{S}^{N-1}$. Therefore, by Fubini's theorem,
    \begin{align*}
        0 ={}& \int_{\R^N} \chi_{\{\nabla v(x)\neq 0\}} \int_{\mathbb{S}^{N-1}} \chi_{\{\nabla v(x)\cdot \tau =0\}}\,d\mathcal{H}^{N-1}\,dx \\
        ={}& \int_{\mathbb{S}^{N-1}}|\{x\in\R^N: \nabla v(x)\neq 0,\,\nabla v(x)\cdot \tau=0\}|\,d\mathcal{H}^{N-1}.
    \end{align*}
    Consequently, $|A_v|=|\mathbb{S}^{N-1}|$.

    Now fix $\tau\in A_v$. Without loss of generality, we may assume $\tau=e_N$ and write $x=(x',y)\in \R^{N-1}\times \R$. We can see from
    \[ |\{x:v_y(x)=0,\, \nabla v(x)\neq 0\}|=0,\]
    that almost every point with $\nabla v \neq 0$ belongs to a region where $v_y\neq 0$.

    By Proposition \ref{Steiner prop}(6), $\phi(\lambda)$ is decreasing, so we only need to show its convexity. Let
    \[ K:= \{(x',v_0)\in \R^{N-1}\times \R: E_{v_0}(x'):=\{y:v(x',y)>v_0\}\neq \varnothing \}.\]
    Since $\{v>v_0\}$ is convex, every nonempty section $E_{v_0}(x')$ is an interval. Hence, for $(x',v_0)\in K$, we can write
    \[ E_{v_0}(x') = (y_1(x',v_0),y_2(x',v_0)), \quad y_1(x',v_0) < y_2(x',v_0).\]
    Thus, for almost every $(x',v_0)\in K$, the corresponding points $z_k:=(x',y_k(x',v_0))$, $k=1,2$ satisfy
    \[ \nabla v(z_k)\neq 0,\quad v_y(z_k)\neq 0,\]
    except for a set which contributes zero to the $p$-Dirichlet energy. Indeed, by $e_N\in A_v$,
    \[ |\{ \nabla v\neq 0,\, v_y=0\}|=0,\]
    while the set $\{\nabla v=0\}$ contributes nothing to $\int_{\R^N} |\nabla v|^p\,dx$. It is clear that $z_k \in \{0<v<\lVert v\rVert_{L^\infty(\Omega)}\}$, and so $v$ is $C^1$ near $z_k$, and $v_y(z_k)\neq 0$, $k=1,2$. Therefore, the implicit function theorem shows that there is an open neighborhood $V\subset K$ of $(x',v_0)$ such that the equation $v(x',y)=v_0$ has exactly two solutions $y=y_k(x',v_0)$ in $V$, where $y_k\in C^1(V)$, $k=1,2$, and $\partial y_1/\partial v_0 > 0 > \partial y_2/\partial v_0$.
    
    Consequently, by changing variables $(x',y) \mapsto (x',v_0=v(x',y))$ and observing that $\{\nabla v=0\}$ contributes zero to the Dirichlet energy, we obtain
    \begin{align*}
        \phi(\lambda) ={}& \sum_{k=1}^2 \int_K \Big| (\partial y_k^t/\partial v)^{-1}(1+\sum_{i=1}^{N-1} (\partial y_k^t/ \partial x_i)^2)^{1/2}\Big|^{p}|\partial y_k^t/\partial v|,
    \end{align*}
    where 
    \begin{equation}\label{y12t}
        \begin{cases}
            y_1^t = \frac 12 (y_1-y_2 + e^{-t}(y_1+y_2)) = (1-\lambda)y_1 -\lambda y_2,\\
            y_2^t = \frac 12 (y_2-y_1 + e^{-t}(y_1+y_2)) = -\lambda y_1 + (1-\lambda)y_2.
        \end{cases}
    \end{equation}
    Set $a(\lambda):=(1+\sum_{i=1}^{N-1} (\partial y_1^t/ \partial x_i)^2)^{1/2}$ and $b(\lambda):=|\partial y_1^t/\partial v|$. Then
    \[ b(\lambda) = (1-\lambda)\frac{\partial y_1}{\partial v} - \lambda \frac{\partial y_2}{\partial v}\]
    is linear in $\lambda$. Moreover, $a(\lambda)$ is convex in $\lambda$. Hence, for $\lambda_1,\lambda_2\in(0,1/2)$ and $\xi \in(0,1)$, the convexity and monotonicity of $x^p$ give
    \begin{align*}
        &\int_K \xi b(\lambda_1)\bigl|\frac{a(\lambda_1)}{b(\lambda_1)}\bigr|^p + (1-\xi)b(\lambda_2)\bigl|\frac{a(\lambda_2)}{b(\lambda_2)}\bigr|^p\\
        &\ge \int_K \big(\xi b(\lambda_1)+(1-\xi)b(\lambda_2)\big)\bigl|\frac{\xi a(\lambda_1) + (1-\xi)a(\lambda_2)}{\xi b(\lambda_1)+(1-\xi)b(\lambda_2)}\bigr|^p\\
        & = \int_K b(\xi\lambda_1 + (1-\xi)\lambda_2)\bigl|\frac{\xi a(\lambda_1) + (1-\xi)a(\lambda_2)}{b(\xi\lambda_1+(1-\xi)\lambda_2)}\bigr|^p \\
        & \ge \int_K b(\xi\lambda_1 + (1-\xi)\lambda_2)\bigl|\frac{a(\xi\lambda_1 + (1-\xi)\lambda_2)}{b(\xi\lambda_1+(1-\xi)\lambda_2)}\bigr|^p.
    \end{align*}
    As a result,
    \begin{align*}
        & \xi \phi(\lambda_1) + (1-\xi)\phi(\lambda_2) \\
        & = \int_K \xi b(\lambda_1)\bigl|\frac{a(\lambda_1)}{b(\lambda_1)}\bigr|^p + (1-\xi)b(\lambda_2)\bigl|\frac{a(\lambda_2)}{b(\lambda_2)}\bigr|^p \\
        &\qquad + \xi b(1-\lambda_1)\bigl|\frac{a(1-\lambda_1)}{b(1-\lambda_1)}\bigr|^p + (1-\xi)b(1-\lambda_2)\bigl|\frac{a(1-\lambda_2)}{b(1-\lambda_2)}\bigr|^p\\
        & \ge \int_K b(\xi\lambda_1 + (1-\xi)\lambda_2)\bigl|\frac{a(\xi\lambda_1 + (1-\xi)\lambda_2)}{b(\xi\lambda_1+(1-\xi)\lambda_2)}\bigr|^p \\
        & \qquad + b(1-(\xi\lambda_1 + (1-\xi)\lambda_2))\bigl|\frac{a(1-(\xi\lambda_1 + (1-\xi)\lambda_2))}{b(1-(\xi\lambda_1+(1-\xi)\lambda_2))}\bigr|^p \\
        & = \phi(\xi\lambda_1 + (1-\xi)\lambda_2).\qedhere
    \end{align*}
\end{proof}
\begin{remark}\label{rmk1}
    This convexity argument requires the one-dimensional sections of all superlevel sets to be intervals. Otherwise, $v(x',y)=v_0$ may have $2m$ solutions $y_k(x',v)$, and when the intervals $[y_{2k-1}^t,y_{2k}^t]$, $k=1,\dots,m$, coalesce during continuous Steiner symmetrization, $\phi(\lambda)$ may jump and fail to be convex.
\end{remark}

Since $A_u$ has full measure on $\mathbb{S}^{N-1}$, after a rotation we may assume that $e_N\in A_u$. Under condition (c), Lemma \ref{good convex} then gives the required convexity of the Dirichlet energy of $u_{\gamma,s}$ along the continuous symmetrization.

\begin{proposition}\label{Prop 1}
    Suppose conditions (a)--(d) are satisfied, and $\epsilon_0<\min\{\mathbf{c}/4,1\}$. Set $C_0:=\lVert u\rVert_{C^1(\bar\Omega)}$ and $C_1:=\lVert f\rVert_{C^0([0,C_0])}$. Then there exists $C_3>0$ depending only on $C_0$, $C_1$, $L$, $\mathbf{c}$, $N$, $p$, and $\diam(\Omega)$ such that, for any $s\in (0,s_*)$ and $0<\gamma < \min\{\epsilon_0^{1/\alpha},s\}$, the following estimate holds:
    \begin{equation}\label{Steiner est}
        \int_{\R^N} |\nabla u_{\gamma,s}|^p \,dx- \int_{\R^N}|\nabla (u_{\gamma,s})^{\infty}|^p\,dx \le C_3\epsilon_0.
    \end{equation}
\end{proposition}

\begin{proof}
    When $\epsilon_0=0$, by Serrin's theorem and \cite{Cao2025}, the left side of \eqref{Steiner est} is zero and the statement holds. Thus, tt suffices to consider $\epsilon_0>0$. This follows from the method of Cao, Wei, and Zhan \cite{Cao2025}. For completeness, we give the full proof.

    Since $\epsilon_0<\mathbf{c}/4$, we have
    \begin{equation}\label{c quarter prop}
        \mathbf{c}/2-\epsilon_0>\mathbf{c}/4.
    \end{equation}

    \noindent \textbf{Step one.} As in Lemma 3.5 of \cite{Cao2025}, we show that
    \begin{equation}\label{est f1}
        \int_{\Omega} f(u) [(u-\gamma)_+^t-(u-\gamma)_+] \,dx \ge -2C_1C_0\diam(\Omega)t |\{0<u<\gamma\}| + o(t),
    \end{equation}
    where $o(t)$ denotes a quantity satisfying $o(t)/t\to 0$ as $t\to 0$.
    
    It suffices to consider $\gamma\in(0,C_0)$. By the definition of $C_0$, $(u-\gamma)_+$ has Lipschitz constant at most $C_0$. We show first that $(u-\gamma)^t_+-(u-\gamma)_+\in W_0^{1,p}(\Omega)$. Clearly, $(u-\gamma)_+ \in W_0^{1,p}(\Omega)$. By Proposition \ref{Steiner prop}(6), for $x\in \Omega^c$, where $u(x)=0$,
    \[ |u^t(x)| \le C_0 \diam(\Omega) t < \gamma,\]
    showing that $(u-\gamma)^t_+(x) = (u^t-\gamma)_+(x)=0$. Thus $(u-\gamma)^t_+ \in W_0^{1,p}(\Omega)$.
    
    Again by Proposition \ref{Steiner prop} (6),
    \begin{equation}\label{Lip u ut}
        |(u(x)-\gamma)^t_+ - (u(x)-\gamma)_+| \le C_0\diam(\Omega)t, \quad \forall\, x\in\R^N,\, \gamma>0,\, t>0.
    \end{equation}
    Thus,
    \begin{align*}
        & \int_{\Omega} f(u) [(u-\gamma)_+^t-(u-\gamma)_+]\,dx\\
        & = \int_{\Omega} [f(u)-f((u-\gamma)_+ + \gamma)] [(u-\gamma)_+^t-(u-\gamma)_+]\,dx \\
        &\quad + \int_{\Omega} f((u-\gamma)_+ +\gamma) [(u-\gamma)_+^t-(u-\gamma)_+]\,dx \\
        &=: I_1(t)+ I_2(t).
    \end{align*}
    For $I_1(t)$, by \eqref{Lip u ut} and $|f(x)|\le C_1$ for $x\in [0,C_0]$,
    \[ |I_1(t)| \le 2C_1C_0\diam(\Omega) t |\{0<u<\gamma\}|.\]
    For $I_2(t)$, set $F(z) = \int_0^z f(x+\gamma)\,dx$ and $v=(u-\gamma)_+$. By Proposition \ref{Steiner prop}(3),
    \begin{align*}
        0 ={}& \int_{\Omega} F(v^t)\,dx - \int_{\Omega} F(v)\,dx \\
        ={}& \int_{\Omega}\bigl[\int_0^1 f\Big(v(x) +\gamma + \theta \big(v^t(x)-v(x)\big) \Big)\,d\theta\bigr] (v^t(x)-v(x))\,dx.
    \end{align*}
    Thus, by $f$ continuous and \eqref{Lip u ut},
    \begin{align*}
        I_2(t) ={}& \int_{\Omega} f(v+\gamma)(v^t-v)\,dx \\
        ={}& \int_{\Omega}\bigl[\int_0^1 \bigl(f(v+\gamma) - f\Big(v +\gamma + \theta \big(v^t-v\big) \Big)\bigr)\,d\theta\bigr] (v^t-v)\,dx \\
        \ge{}& -\sup\bigl\{ |f(v+\gamma) - f(v+\gamma + \theta(v^t-v))| : x\in\Omega,\, \theta\in[0,1]\bigr\} \cdot C_0 \diam(\Omega)t\\
        \ge{}& o(t), \quad t\to 0.
    \end{align*}
    Combining the two estimates gives \eqref{est f1}. When $\gamma>C_0\diam(\Omega)t$, we may therefore test \eqref{Eq1} with $(u-\gamma)^t_+-(u-\gamma)_+$ and obtain
    \begin{align*}
        \int_{\Omega} f(u)[(u-\gamma)_+^t-(u-\gamma)_+]\,dx ={}& \int_{\Omega} |\nabla u|^{p-2}\nabla u \cdot \nabla\big( (u^t-\gamma)_+-(u-\gamma)_+\big)\,dx\\
        ={}& \int_{\{u^t>\gamma\}} |\nabla u|^{p-2} \nabla u\cdot \nabla u^t \,dx - \int_{\{u>\gamma\}} |\nabla u|^p\,dx \\
        \le{}& \int_{\{u^t>\gamma\}} |\nabla u|^{p-1}|\nabla u^t| \,dx - \int_{\{u>\gamma\}} |\nabla u|^p\,dx.
    \end{align*}
    Hence,
    \begin{equation}\label{est 1}
        \int_{\{u^t>\gamma\}} |\nabla u|^{p-1}|\nabla u^t|\,dx - \int_{\{u>\gamma\}} |\nabla u|^{p}\,dx \ge -2C_1C_0 \diam(\Omega) t |\{0<u<\gamma\}| + o(t).
    \end{equation}

    \noindent \textbf{Step two.} We use the method of \cite[Lemma 3.4]{Cao2025} to derive
    \begin{equation}\label{est 2}
        \begin{aligned}
            & \int_{\{u^t>\gamma\}} |\nabla u|^p \,dx- \int_{\{u>\gamma\}} |\nabla u|^p\,dx \\
            & \le p 2^{2p}\mathbf{c}^{-1}\bigl(L(\frac{2}{\mathbf{c}}(C_0\diam(\Omega)t+\gamma))^{\alpha}+\epsilon_0\bigr)C_0C_2\diam(\Omega)t,
        \end{aligned}
    \end{equation}
    Indeed, by the definition of $C_0$, the Lipschitz constant of $u$ is at most $C_0$. Since $\Omega\subset B_{\diam(\Omega)}$,
    \[|u(x)-u^t(x)|\le C_0\diam(\Omega)t,\quad \forall\, x\in\Omega,\, t\in [0,\infty].\]
    For $\gamma> C_0\diam(\Omega)t$, we have 
    \begin{align*}
        \{u>\gamma\}\Delta \{u^t>\gamma\}\subset{}& \{\gamma-C_0\diam(\Omega)t\le u \le \gamma+ C_0\diam(\Omega)t\}.
    \end{align*}
    By condition (b) and Corollary \ref{Cor1} (2), for all
    \[ x\in \{\gamma-C_0\diam(\Omega)t\le u \le \gamma+C_0\diam(\Omega)t\},\]
    there holds
    \begin{align*}
        ||\nabla u(x)|-\mathbf{c}| \le L d(x,\partial\Omega)^{\alpha} + \epsilon_0 \le L(\frac{2}{\mathbf{c}}(C_0\diam(\Omega)t+\gamma))^{\alpha}+\epsilon_0 < \mathbf{c}
    \end{align*}
    provided
    \[ C_0\diam(\Omega) t < \gamma < \bigl(\frac{\mathbf{c}}{4L}\bigr)^{1/\alpha}\frac{\mathbf{c}}2 - C_0\diam(\Omega)t,\]
    in view of \eqref{c quarter prop}.
    
    Proposition \ref{Steiner prop}(1) shows that $|\{u>\gamma\}| = |\{u^t>\gamma\}|$. By \eqref{c quarter prop} and Corollary \ref{Cor1}(3), $\{u=c\}$ is locally $C^1$ and has finite perimeter for all
    \[0<c<\bigl(\frac{\mathbf{c}}{4L}\bigr)^{1/\alpha}\frac{\mathbf{c}}{2}.\]
    Hence, using the elementary inequality
    \[ |a^p-b^p|\le p\max\{a,b\}^{p-1}|a-b|, \quad \forall\, a,b>0,\, p>1,\] 
    and co-area formula, we derive that 
    \begin{align*}
        & \int_{\{u^t>\gamma\}} |\nabla u|^p\,dx - \int_{\{u>\gamma\}} |\nabla u|^p\,dx \\
        & = \int_{\{u^t>\gamma\}} (|\nabla u|^p-\mathbf{c}^p) \,dx - \int_{\{u>\gamma\}} (|\nabla u|^p-\mathbf{c}^p)\,dx \\
        & \le \int_{\{\gamma-C_0\diam(\Omega)t \le u \le \gamma +C_0\diam(\Omega)t\}} ||\nabla u|^p - \mathbf{c}^p|\,dx \\
        & \le \int_{\gamma-C_0\diam(\Omega)t}^{\gamma +C_0\diam(\Omega)t} \int_{\{u=y\}} |\nabla u|^{-1} p (2\mathbf{c})^{p-1}||\nabla u|-\mathbf{c}|\,d\mathcal{H}^{N-1}\,dy\\
        & \le p(2\mathbf{c})^{p-1} \bigl(L(\frac{2}{\mathbf{c}}(C_0\diam(\Omega)t+\gamma))^{\alpha}+\epsilon_0\bigr) \\
        &\qquad \cdot \int_{\gamma-C_0\diam(\Omega)t}^{\gamma +C_0\diam(\Omega)t}\int_{\{u=y\}}(2/\mathbf{c})^p|\nabla u|^{p-1}\,d\mathcal{H}^{N-1}dy \\
        & \le p 2^{2p}\mathbf{c}^{-1} \bigl(L(\frac{2}{\mathbf{c}}(C_0\diam(\Omega)t+\gamma))^{\alpha}+\epsilon_0\bigr)C_0C_2\diam(\Omega)t.
    \end{align*}

    \noindent \textbf{Step three.} By \cite[Lemma 3.2]{Cao2025}, if $G(x)=\int_0^x g(y)\,dy$ with $g\ge 0$ strictly increasing, and $h(y)=yg(y)-G(y)$, then for $\gamma>0$,
    \begin{equation}\label{Lem3.2 in Wei}
        \begin{aligned}
            &\int_{\{u^t>\gamma\} } G(|\nabla u^t|)\,dx - \int_{ \{u>\gamma\} }G(|\nabla u|)\,dx \\
            &\ge \int_{\{u^t>\gamma\}} g(|\nabla u|)|\nabla u^t|\,dx - \int_{\{u>\gamma\}} g(|\nabla u|)|\nabla u|\,dx\\
            &\qquad- \bigl(\int_{\{u^t>\gamma\}} h(|\nabla u|)\,dx - \int_{\{u>\gamma\}} h(|\nabla u|)\,dx\bigr).
        \end{aligned}
    \end{equation}
    Let $G(x)=x^p$, $g(x)=px^{p-1}$, $h(x)=(p-1)x^p$, using \eqref{est 1} and \eqref{est 2}, we obtain that 
    \begin{align*}
        & \int_{\{u^t>\gamma\}} |\nabla u^t|^p\,dx -\int_{\{u>\gamma\}}|\nabla u|^p\,dx \\
        &\ge p\bigl(\int_{\{u^t>\gamma\}} |\nabla u|^{p-1}|\nabla u^t|\,dx - \int_{\{u>\gamma\}} |\nabla u|^{p}\,dx\bigr) \\
        &\qquad - (p-1)\bigl(\int_{\{u^t>\gamma\}} |\nabla u|^{p}\,dx- \int_{\{u>\gamma\}} |\nabla u|^p\,dx\bigr) \\
        & \ge \Big(-2pC_1C_0\diam(\Omega) |\{0<u<\gamma\}| \\
        & \qquad - p(p-1) 2^{2p}\mathbf{c}^{-1}\big(L(\frac{2}{\mathbf{c}}(C_0\diam(\Omega)t+\gamma))^{\alpha}+\epsilon_0\big)C_0C_2\diam(\Omega)\Big)t + o(t),
    \end{align*}
    provided
    \[C_0\diam(\Omega)t < \gamma < \bigl(\frac{\mathbf{c}}{4L}\bigr)^{1/\alpha}\frac{\mathbf{c}}2-C_0\diam(\Omega)t.\]

    Note that for all $s\in (\gamma,C_0)$, by Lemma \ref{commute}, we have $(u-\gamma)_+ = u_{\gamma,s} + (u-s)_+$, $(u^t-\gamma)_+ = (u_{\gamma,s})^t + (u-s)_+^t$. Hence,
    \begin{align*}
        & \int_{\{u^t>\gamma\}} |\nabla u^t|^p \,dx - \int_{\{u>\gamma\}}|\nabla u|^p \,dx \\
        & = \int_{\R^N} |\nabla (u_{\gamma,s})^t|^p \,dx -\int_{\R^N} |\nabla u_{\gamma,s}|^p \,dx + \int_{\R^N} |\nabla (u-s)_+^t|^p \,dx - \int_{\R^N}|\nabla (u-s)^+|^p \,dx \\
        & \le \int_{\R^N} |\nabla (u_{\gamma,s})^t|^p \,dx - \int_{\R^N}|\nabla u_{\gamma,s}|^p \,dx.
    \end{align*}
    Moreover, by Corollary \ref{Cor1} (3) and the co-area formula,
    \begin{align*}
        |\{0<u<\gamma\}| ={}& \int_0^\gamma \int_{\{u=y\}} |\nabla u|^{-1}\,d\mathcal{H}^{N-1}\,dy \\
        \le{}& \int_0^\gamma \int_{\{u=y\}} (2/\mathbf{c})^p |\nabla u|^{p-1} \,d\mathcal{H}^{N-1}\,dy \\
        \le{}& (2/\mathbf{c})^{p} C_2 \gamma.
    \end{align*}
    Thus, letting $t\to 0$, we deduce that, for
    \[0<\gamma<\bigl(\frac{\mathbf{c}}{4L}\bigr)^{1/\alpha}\frac{\mathbf{c}}2,
      \qquad \gamma<s<C_0,\]
    there holds
    \begin{align}\label{est small t}
        & \liminf_{t\to 0} \frac{1}{t}\bigl(\int_{\R^N} |\nabla (u_{\gamma,s})^t|^p \,dx -\int_{\R^N}|\nabla u_{\gamma,s}|^p \,dx \bigr) \notag \\
        & \ge -2p C_0C_1 \diam(\Omega) |\{0<u<\gamma\}| - p(p-1) 2^{2p}\mathbf{c}^{-1}(L\gamma^\alpha + \epsilon_0)C_0C_2\diam(\Omega) \notag \\
        & \ge - p C_0C_2\diam(\Omega) \big(2 C_1 (2/\mathbf{c})^p \gamma + (p-1) 2^{2p}\mathbf{c}^{-1}(L\gamma^\alpha + \epsilon_0)\big).
    \end{align}

    \noindent \textbf{Step four.} By condition (c), $\{u>c\}$ is convex for every $0\le c\le s_*$. Thus, for $s\in (0,s_*)$, $\{u_{\gamma,s}>c\} = \{u>c+\gamma\}$ is convex for every $c\in(0,s-\gamma)$. By Lemma \ref{good convex}, because $\nabla u_{\gamma,s} = \chi_{\{\gamma<u<s\}}\nabla u$ almost everywhere and so $e_N\in A_u\subset A_{u_{\gamma,s}}$, we see that
    \[ \phi(\lambda):= \int_{\R^N} |\nabla (u_{\gamma,s})^t|^p\,dx, \quad \lambda=\frac{1-e^{-t}}{2} \]
    is decreasing and convex for $\lambda\in (0,1/2)$. As a result, for $\lambda_0\to 0_+$ (then the related $t_0\to 0_+$), by \eqref{est small t}
    \begin{align*}
        \frac{\phi(0) - \phi(1/2)}{1/2} \le{}& \frac{\phi(0)-\phi(\lambda_0)}{\lambda_0}\\
        ={}& \frac{2t_0}{1-e^{-t_0}}\frac{1}{t_0}\bigl(\int_{\R^N} |\nabla u_{\gamma,s}|^p \,dx - \int_{\R^N} |\nabla (u_{\gamma,s})^{t_0}|^p\,dx\bigr) \\
        \le{}& 2p C_0C_2\diam(\Omega) \big(2 C_1 (2/\mathbf{c})^p \gamma + (p-1) 2^{2p}\mathbf{c}^{-1}(L\gamma^\alpha + \epsilon_0)\big).
    \end{align*}
    If $\gamma <\min\{\epsilon_0^{1/\alpha},s\}$, it follows that \eqref{Steiner est} holds with
    \[ C_3=4pC_0C_1|\Omega|\diam(\Omega) \big( 2C_1(2/\mathbf{c})^p + (p-1)2^{2p}\mathbf{c}^{-1} (L+1) \big). \qedhere\]
\end{proof}

Thanks to Proposition \ref{Steiner prop} (6) and Proposition \ref{Prop 1}, we can then derive a stability result for Steiner symmetrization. That is to say, for a.e. $\tau \in \mathbb{S}^{N-1}$, $\Omega$ is ``nearly'' symmetric with respect to some hypersurface with normal vector $\tau$.

\begin{theorem}\label{thm1}
    Assume conditions (a)--(d) are satisfied, and set $C_0:=\lVert u\rVert_{C^1(\bar\Omega)}$ and $C_1:=\lVert f\rVert_{C^0([0,C_0])}$. Then, when $\epsilon_0 < \min\{\mathbf{c}/4,1\}$,
    \begin{equation}\label{est Omega final}
        \inf_{y\in\R} |\Omega\Delta (\Omega^{\infty}+(0,y))| \le C_9\epsilon_0^{\frac{\alpha}{2(\alpha+1)}} ,
    \end{equation}
    where $C_9$ is a constant depending on $\diam(\Omega), \rho_i,N,p,\mathbf{c}$, $C_0$, $C_1$, $L$, $\alpha$, and $s_*$.
\end{theorem}
\begin{proof}
    Since $\epsilon_0<\mathbf{c}/4$, by \eqref{c quarter prop}, we may fix
    \[ 0<s<\min\bigl\{\bigl(\frac{\mathbf{c}}{4L}\bigr)^{1/\alpha}\frac{\mathbf{c}}{2},s_*\bigr\},
      \qquad 0<\gamma<\min\{\epsilon_0^{1/\alpha},s/2\}.\]
    
    For $x\in \{\gamma<u<s\}$, Corollary \ref{Cor1}(2) gives $d(x,\partial\Omega) \le 2s/\mathbf{c}$, and hence
    \begin{equation}\label{est nabla u}
        \begin{aligned}
            ||\nabla u_{\gamma,s}(x)|-\mathbf{c}| \le{}& |\nabla u(x) - \nabla u(y)| + ||\nabla u(y)|-\mathbf{c}|\\
            \le{}& L|x-y|^{\alpha} + \epsilon_0 \\
            \le{}& L(2s/\mathbf{c})^{\alpha} + \epsilon_0.
        \end{aligned}
    \end{equation}
    On the other hand, by Proposition \ref{Steiner prop} (6), the Lipschitz constant of $u_{\gamma,s}^\infty$ is less than that of $u_{\gamma,s}$, giving that 
    \begin{equation}\label{est nabla us}
        |\nabla (u_{\gamma,s})^\infty(x)| \le \lVert \nabla u_{\gamma,s}\rVert_{L^\infty(\Omega)} \le \mathbf{c} + L(2s/\mathbf{c})^\alpha + \epsilon_0,\quad  \text{for a.e. $x\in\Omega$}.
    \end{equation}
    Set $E_y:=\{u>y\}$ and $E_y^\infty:=\{u^\infty>y\}$. By the coarea formula and \eqref{est nabla u}--\eqref{est nabla us},
    \begin{align}
        \int_{\R^N} |\nabla u_{\gamma,s}|^p\,dx = {}& \int_\gamma^s \int_{\{u=y\}} |\nabla u|^{p-1}\,d \mathcal{H}^{N-1}\,dy \notag \\
        \ge{}& (\mathbf{c}-L(2s/\mathbf{c})^{\alpha}-\epsilon_0)^{p-1} \int_\gamma^s P(E_y)\,dy; \label{est 3}
    \end{align}
    Similarly,
    \begin{align}
        \int_{\R^N}|\nabla (u_{\gamma,s})^{\infty}|^p\,dx ={}& \int_\gamma^s \int_{\{u^\infty = y\}}|\nabla u^\infty|^{p-1}\,d \mathcal{H}^{N-1}\,dy \notag\\
        \le{}& (\mathbf{c} + L(2s/\mathbf{c})^\alpha + \epsilon_0)^{p-1}\int_\gamma^s P(E_y^\infty)\,dy. \label{est 4}
    \end{align}
    Thus,
    \begin{align*}
        &\int_{\R^N} |\nabla u_{\gamma,s}|^p \,dx- \int_{\R^N}|\nabla (u_{\gamma,s})^{\infty}|^p\,dx\\
        &\ge (\mathbf{c}-L(2s/\mathbf{c})^{\alpha}-\epsilon_0)^{p-1} \int_\gamma^s P(E_y)\,dy - (\mathbf{c} + L(2s/\mathbf{c})^\alpha + \epsilon_0)^{p-1} \int_\gamma^s P(E_y^\infty) \,dy\\
        & \ge (\mathbf{c} + L(2s/\mathbf{c})^\alpha + \epsilon_0)^{p-1} \int_\gamma^s (P(E_y)-P(E_y^\infty))\,dy\\
        & \quad - C_5\mathbf{c}^{p-2} (L(2s/\mathbf{c})^{\alpha}+\epsilon_0)\int_\gamma^s P(E_y)\,dy,
    \end{align*}
    Here $C_5 = 2(p-1)\max\{(1/2)^{p-2},(3/2)^{p-2}\}$ satisfies
    \[ (\mathbf{c} + L(2s/\mathbf{c})^\alpha + \epsilon_0)^{p-1} - (\mathbf{c}-L(2s/\mathbf{c})^{\alpha}-\epsilon_0)^{p-1} \le C_5 \mathbf{c}^{p-2}(L(2s/\mathbf{c})^{\alpha}+\epsilon_0).\]
    By Corollary \ref{Cor1} (3), $1 \le (2/\mathbf{c})|\nabla u|$ and so 
    \[ \int_\gamma^s P(E_y)\,dy \le \int_\gamma^s \int_{\{u=y\}} (2/\mathbf{c})^{p-1}|\nabla u|^{p-1}\,d\mathcal{H}^{N-1}\,dy \le (2/\mathbf{c})^{p-1}C_2 s.\]
    Applying \eqref{Steiner est}, we obtain the following estimate for the integral of the perimeter deficit:
    \begin{equation*}
        \int_\gamma^s (P(E_y)-P(E_y^\infty))\,dy \le C_6(\epsilon_0+ s(s^{\alpha}+\epsilon_0) ),
    \end{equation*}
    \[ C_6 = \max\bigl\{\frac{C_3}{\mathbf{c}^{p-1}}, \frac{C_5C_2}{\mathbf{c}}\bigl(\frac{2}{\mathbf{c}}\bigr)^{p-1}, \frac{C_5C_2L}{\mathbf{c}}\bigl(\frac{2}{\mathbf{c}}\bigr)^{\alpha+p-1}\bigr\}.\]
    Since $\gamma<s/2$, in particular, there is some $s_0\in(s/2,s)$ such that
    \begin{equation}\label{est perimeter}
        P(E_{s_0})-P(E_{s_0}^\infty) \le 2C_6(s^{\alpha} + \epsilon_0 + \epsilon_0/s).
    \end{equation}
    Since $E_{s_0}$ is convex, the result of Barchiesi, Cagnetti, and Fusco \cite{Barchiesi2013} gives a constant $C_7=C_7(N)$ such that
    \begin{equation}\label{steiner stab}
        \inf\bigl\{ \frac{|E_{s_0}\Delta (E_{s_0}^{\infty}+(0,y))|}{|E_{s_0}|}:y\in\R\bigr\} \le C_7 \mathcal{E}^2 \sqrt{\frac{P(E_{s_0})-P(E_{s_0}^{\infty})}{P(E_{s_0}^\infty)}},
    \end{equation}
    where $\mathcal{E}$ denotes the ratio of the circumradius to the inradius of $E_{s_0}$.

    Now we connect this estimate with $\Omega$. By Corollary \ref{Cor1} (2)(3),
    \begin{align*}
        |\Omega\backslash E_{s_0}| = |\{0< u\le s_0\}| ={}& \int_0^{s_0}\int_{\{u=y\}} |\nabla u|^{-1}\,d\mathcal{H}^{N-1}(x)\,dy\\
        \le{}& (2/\mathbf{c})^{p}\int_0^{s_0} \int_{\{u=y\}}|\nabla u|^{p-1} \,d\mathcal{H}^{N-1}(x)\,dy \\
        \le{}& (2/\mathbf{c})^{p} C_2s.
    \end{align*}
    Therefore, for any $(0,y)\in \R^{N-1}\times \R$,
    \begin{align*}
        |\Omega\Delta (\Omega^{\infty}+(0,y))|={}& 2|\Omega\backslash (\Omega^\infty+(0,y))| \\
        \le{}& 2|\Omega\backslash E_{s_0}|+2|E_{s_0}\backslash (E_{s_0}^\infty + (0,y))| \\
        ={}& 2|\Omega\backslash E_{s_0}|+ |E_{s_0}\Delta (E_{s_0}^\infty + (0,y))|.
    \end{align*}
    Optimizing over all $(0,y)$ and using \eqref{est perimeter},
    \begin{equation}
        \inf_{y\in\R} |\Omega\Delta (\Omega^{\infty}+(0,y))| \le \sqrt{2}C_6^{1/2}C_7\mathcal{E}^2 |E_{s_0}|\sqrt{\frac{s^{\alpha} + \epsilon_0 + \epsilon_0/s}{P(E_{s_0}^\infty)}} + 2(2/\mathbf{c})^{p}C_2s.
    \end{equation}
    Moreover, the isoperimetric inequality gives a bound for $\frac{|E_{s_0}|}{\sqrt{P(E_{s_0}^\infty)}}$:
    \begin{equation*}
        \begin{gathered}
            P(E_{s_0}^\infty) \ge N w_N^{1/N} |E_{s_0}^\infty|^{(N-1)/N} = N w_N^{1/N} |E_{s_0}|^{(N-1)/N} \\
            \implies \frac{|E_{s_0}|}{\sqrt{P(E_{s_0}^\infty)}} \le |\Omega|^{\frac{N+1}{2N}}(N w_N^{1/N})^{-1/2}
        \end{gathered}
    \end{equation*}
    and so when $s<1$,
    \begin{equation}\label{est Omega}
        \inf_{y\in\R} |\Omega\Delta (\Omega^{\infty}+(0,y))| \le C_8 \mathcal{E}^2 ( \epsilon_0^{1/2}/s^{1/2} + s^{\alpha/2}) +2(2/\mathbf{c})^{p}C_2s^{\alpha/2},
    \end{equation}
    \[ C_8 = 2C_6^{1/2}C_7 |\Omega|^{\frac{N+1}{2N}}(Nw_N^{1/N})^{-1/2}.\]
    Denote the circumradius and inradius of $\Omega$ by $\rho_e$ and $\rho_i$, respectively. Then $\rho_e\le \diam(\Omega)$. By Corollary \ref{Cor1}(2), $\Omega_{2s_0/\mathbf{c}}\subset E_{s_0}$, and hence
    \[ \mathcal{E}\le \frac{\rho_e}{\rho_i - 2s_0/\mathbf{c}}\le 2\rho_e/\rho_i \le 2\diam(\Omega)/\rho_i\]
    provided $\rho_i \ge 4s_0/\mathbf{c}$.

    Choose
    \[s = \min\bigl\{s_*, \bigl(\frac{\mathbf{c}}{4L}\bigr)^{1/\alpha}\frac{\mathbf{c}}{2}, \frac{\rho_i\mathbf{c}}{4}, 1\bigr\}\epsilon_0^{\frac{1}{\alpha+1}}.\]
    Since $\epsilon_0<\min\{1,\mathbf{c}/4\}$, the hypotheses needed to apply Corollary \ref{Cor1} are satisfied. Moreover,
    \[s_0\le s \le \rho_i\mathbf{c}/4,\quad s<1,\] 
    and so \eqref{est Omega} implies
    \begin{equation*}
        \inf_{y\in\R} |\Omega\Delta (\Omega^{\infty}+(0,y))| \le C_9\epsilon_0^{\frac{\alpha}{2(\alpha+1)}},
    \end{equation*}
    where
    \[ C_9 = C_8\bigl(\frac{2\diam(\Omega)}{\rho_i}\bigr)^2\Big( A^{-1/2} + A^{\alpha/2}\Big) + 2(2/\mathbf{c})^{p}C_2A^{\alpha/2},\]
    and
    \[A=\min\bigl\{s_*, \bigl(\frac{\mathbf{c}}{4L}\bigr)^{1/\alpha}\frac{\mathbf{c}}{2}, \frac{\rho_i\mathbf{c}}{4}, 1\bigr\}.\]
    
    When $\epsilon_0 < \min\{1,\mathbf{c}/4\}$, checking all the constants, we see that $C_9$ depends only on $\diam(\Omega)$, $\rho_i$, $N$, $p$, $\mathbf{c}$, $C_0$, $C_1$, $L$, $\alpha$, $s_*$.
\end{proof}

The next step is standard. If $\Omega$ is ``nearly'' symmetric with respect to a.e. direction of reflection, then it is ``nearly'' centrally symmetric.

\begin{proof}[Proof of Theorem \ref{mainthm}]
    Throughout this proof, $C_0:=\lVert u\rVert_{C^1(\bar\Omega)}$ and $C_1:=\lVert f\rVert_{C^0([0,C_0])}$.
    For every direction $\tau\in\mathbb{S}^{N-1}$, denote by $\Omega^\tau$ the Steiner symmetrization with respect to this direction. Since $\Omega$ is bounded, by Theorem \ref{thm1}, for a.e. $\tau\in\mathbb{S}^{N-1}$, there exists $a_\tau \in\R$ such that
    \[ |\Omega \Delta (\Omega^{\tau}+a_\tau \tau)| \le C_9\epsilon_0^{\frac{\alpha}{2(\alpha+1)}} . \]
    $\Omega^\tau+a_\tau \tau$ is symmetric with respect to the hyperplane
    \[ H_\tau := \{ x\in\R^N: x\cdot \tau = a_\tau\}.\]
    Let $P_{\tau,a_\tau}$ be the reflection across $H_\tau$. Then $P_{\tau,a_\tau}(\Omega^\tau+a_\tau \tau) = \Omega^\tau+a_\tau \tau$, and hence
    \begin{align}
        \delta:=|\Omega \Delta P_{\tau,a_\tau}\Omega| \le{}& |\Omega\Delta (\Omega^\tau + a_\tau \tau)| + |P_{\tau,a_\tau}\Omega \Delta(\Omega^\tau + a_\tau \tau)| \notag \\
        ={}&  |\Omega\Delta (\Omega^\tau + a_\tau \tau)| + |P_{\tau,a_\tau}\Omega\Delta  P_{\tau,a_\tau}(\Omega^\tau + a_\tau \tau)| \notag \\
        ={}& 2 |\Omega\Delta (\Omega^\tau + a_\tau \tau)| \notag \\
        \le{}& 2 C_9\epsilon_0^{\frac{\alpha}{2(\alpha+1)}} . \label{est o1}
    \end{align}
    Let $x_\Omega$ be the barycenter of $\Omega$:
    \[ x_\Omega := \frac{1}{|\Omega|}\int_\Omega x\,dx.\]
    Up to a translation, we may assume $x_\Omega=0$. Then the barycenter of $P_{\tau,a_\tau}\Omega$ is $P_{\tau,a_\tau}(0) = 2a_\tau\tau$.

    If $\delta \ge |\Omega|$, then \eqref{main est} is immediate because
    \[ |\Omega \Delta B_\Omega(0)| \le 2|\Omega| \le 2\delta \le 4C_9 \epsilon_0^{\frac{\alpha}{2(\alpha+1)}}.\]
    Otherwise, if $\delta<|\Omega|$, the two sets intersect. Choose $z\in \Omega\cap P_{\tau,a_\tau}\Omega$. Then both $\Omega$ and $P_{\tau,a_\tau}\Omega$ lie in $B_{\diam(\Omega)}(z)$. Thus,
    \begin{equation*}
        \begin{aligned}
            2|a_\tau||\Omega| ={}& \bigl|\int_\Omega x\,dx - \int_{P_{\tau,a_\tau}\Omega} x\,dx\bigr|\\
            ={}& \bigl|\int_\Omega (x-z)\,dx - \int_{P_{\tau,a_\tau}\Omega} (x-z)\,dx\bigr|\\
            \le{}& \diam(\Omega)|\Omega\Delta P_{\tau,a_\tau}\Omega|\\
            \le{}& 2\diam(\Omega) C_9\epsilon_0^{\frac{\alpha}{2(\alpha+1)}}.
        \end{aligned}
    \end{equation*}
    Therefore,
    \begin{equation}\label{est atau}
        |a_\tau| \le 2C_9 |\Omega|^{-1} \diam(\Omega)\epsilon_0^{\frac{\alpha}{2(\alpha+1)}} .
    \end{equation}

    Let $Q_{\tau}$ be the reflection across $\{x\in\R^N: x\cdot \tau = 0\}$. Since $Q_{\tau}\Omega = P_{\tau,a_\tau}\Omega - 2a_\tau \tau$, the standard finite-perimeter estimate $|E\Delta(E+h)|\le |h|P(E)$ gives
    \begin{align*}
        |Q_{\tau}\Omega\Delta P_{\tau,a_\tau}\Omega| \le 2|a_\tau| P(\Omega) \le 2|a_\tau|(2/\mathbf{c})^{p-1} C_2,
    \end{align*}
    here we use the fact that
    \[ P(\Omega)\le \liminf_{y\to 0} P(E_y) \le (2/\mathbf{c})^{p-1}C_2.\]
    Then, by \eqref{est o1} and \eqref{est atau}, we obtain that
    \begin{align}
        |\Omega \Delta Q_{\tau}\Omega| \le{}& |\Omega\Delta P_{\tau,a_\tau}\Omega| + |Q_{\tau}\Omega\Delta P_{\tau,a_\tau}\Omega| \notag \\
        \le{}& 2C_9(1 + 2(2/\mathbf{c})^{p-1}C_2 |\Omega|^{-1}\diam(\Omega))\epsilon_0^{\frac{\alpha}{2(\alpha+1)}}. \label{est o2}
    \end{align}
    For arbitrary $\tau\in\mathbb{S}^{N-1}$, choose $\tau_k\in A_u$ such that $\tau_k \to \tau$. Then $Q_{\tau_k}\to Q_\tau$ as orthogonal matrices. Since the action of $O(N)$ is strongly continuous in $L^1(\R^N)$, we immediately obtain that $\chi_{Q_{\tau_k}\Omega}\to \chi_{Q_\tau\Omega}$ in $L^1(\R^N)$. Consequently, \eqref{est o2} holds for every direction $\tau\in\mathbb{S}^{N-1}$.
    
    Since any rotation $R\in SO(N)$ can be written as at most $N$ reflections and $C_2=C_1|\Omega|$, we have
    \begin{equation}
        |\Omega\Delta R(\Omega)| \le 2NC_9(1 + 2(2/\mathbf{c})^{p-1}C_1\diam(\Omega))\epsilon_0^{\frac{\alpha}{2(\alpha+1)}}.
    \end{equation}

    We now return to the proof of the main theorem. Since $\Omega = \{u>0\}$ is convex and $0$ is its barycenter, $\Omega$ is star-shaped with respect to the origin. Let $\rho:\mathbb{S}^{N-1}\to (0,\infty)$ be its radial function:
    \[ \rho(\theta) = \sup\{r>0:r\theta\in\Omega\},\]
    then
    \[ |\Omega| = \frac{1}{N} \int_{\mathbb{S}^{N-1}}\rho(\theta)^N\,d\theta,\]
    \[ |\Omega\Delta R(\Omega)| = \frac 1N \int_{\mathbb{S}^{N-1}}|\rho(\theta)^N - \rho(R\theta)^N|\,d\theta,\]
    implying that
    \begin{align}
        \int_{\mathbb{S}^{N-1}}|\rho(\theta)^N - \rho(R\theta)^N|\,d\theta \le{}& N |\Omega\Delta R(\Omega)| \notag \\
        \le{}& 2N^2C_9(1 + 2(2/\mathbf{c})^{p-1}C_1\diam(\Omega))\epsilon_0^{\frac{\alpha}{2(\alpha+1)}}.\label{rho est1}
    \end{align}
    Let $\mu_0$ denote the normalized Haar measure on $SO(N)$, and set
    \[ \Phi(\theta):=\rho(\theta)^N,\quad \bar{\Phi}:=\frac{1}{|\mathbb{S}^{N-1}|} \int_{\mathbb{S}^{N-1}}\Phi(\theta)\,d\theta=:r_0^N. \]
    For every fixed $\theta\in\mathbb{S}^{N-1}$, the map $R\mapsto R\theta$ is the normalized surface measure on $\mathbb{S}^{N-1}$. Hence,
    \[ \bar{\Phi} = \int_{SO(N)} \Phi(R\theta)\,d\mu_0(R).\]
    Therefore,
    \begin{align*}
        |\Phi(\theta)-\bar{\Phi}| ={}& \bigl|\Phi(\theta)- \int_{SO(N)} \Phi(R\theta)\,d\mu_0(R) \bigr| \\
        ={}& \bigl|\int_{SO(N)}  (\Phi(\theta) - \Phi(R\theta))\,d\mu_0(R) \bigr|\\
        \le{}& \int_{SO(N)} |\Phi(\theta)-\Phi(R\theta)|\,d\mu_0(R).
    \end{align*}
    Integrating with respect to $\theta$ and using \eqref{rho est1} yields
    \begin{align*}
        \int_{\mathbb{S}^{N-1}}|\rho(\theta)^N - r_0^N|\,d\theta ={}& \int_{\mathbb{S}^{N-1}}|\Phi(\theta)-\bar{\Phi}|\,d\theta \\
        \le{}& \int_{SO(N)}\int_{\mathbb{S}^{N-1}} |\Phi(\theta)-\Phi(R\theta)|\,d\theta\,d\mu_0(R) \\
        \le{}& 2N^2C_9(1 + 2(2/\mathbf{c})^{p-1}C_1\diam(\Omega))\epsilon_0^{\frac{\alpha}{2(\alpha+1)}}.
    \end{align*}
    By the definition of $r_0$, it is exactly the radius of the ball with the same volume as $\Omega$. Hence,
    \[ |\Omega\Delta B_{r_0}(0)| = \frac{1}{N} \int_{\mathbb{S}^{N-1}}|\rho(\theta)^N - r_0^N|\,d\theta \le C_9C'\epsilon_0^{\frac{\alpha}{2(\alpha+1)}}, \]
    with $C'=2N(1 + 2(2/\mathbf{c})^{p-1}C_1\diam(\Omega))$, and this finishes the proof.
\end{proof}

\section{Applications to Elliptic Equations}\label{sec4}
\subsection{The \texorpdfstring{$p$}{p}-Laplacian}
Under additional assumptions on $\Omega$ and $f$, regularity and geometric arguments yield conditions (b)--(d). More precisely, Lemma \ref{regular1} gives the $C^{1,\alpha}$ regularity required in condition (b), while condition (d) follows from the assumptions on the domain. There are two ways to verify condition (c). Under suitable structural assumptions on $f$, Lemma \ref{concave} implies that $\log u$ is concave and hence that the superlevel sets of $u$ are convex. Alternatively, without imposing these structural assumptions on $f$, Lemmas \ref{C^{2,beta}} and \ref{Lem convex} show that the superlevel sets near zero are convex, provided the domain is sufficiently regular and uniformly convex. These observations yield Corollary \ref{maincor1}.

We first introduce notation that quantifies the regularity of $\partial\Omega$.
\begin{definition}\label{character}
    Suppose $\partial\Omega$ is of class $C^{k,\alpha}$. Then there are constants $R_\Omega,M_\Omega>0$ such that, for every $x_0\in\partial\Omega$, one can choose coordinates that send $x_0$ to the origin and $\nu(x_0)$ to the positive $x_N$-axis, in which a local defining function $\varphi_{x_0}$ satisfies
    \[ \Omega\cap Q_{R_\Omega} = \{ (x',y)\in  Q_{R_\Omega}: y>\varphi_{x_0}(x')\}\]
    with 
    \[ \varphi_{x_0} \in C^{k,\alpha}(B_{R_\Omega}'),\quad \varphi_{x_0}(0)=0,\quad \nabla\varphi_{x_0}(0)= 0\]
    and
    \[ \lVert \varphi_{x_0}\rVert_{C^{k,\alpha}(B_{R_\Omega}')}\le M_{\Omega}.\]
    Here $B_R'$ denotes the ball of radius $R$ in $\R^{N-1}$ and $Q_{R}=B_{R}'\times(-R,R)$. We call
    \[ \mathcal{C}_{k,\alpha}(\Omega) = (R_\Omega,M_{\Omega}) \]
    a $C^{k,\alpha}$ character of $\Omega$.
\end{definition}

\begin{lemma}\label{regular1}
    Assume $\Omega$ is of class $C^{1,r}$ and condition (a) holds. Set $C_0:=\lVert u\rVert_{C^1(\bar\Omega)}$ and $C_1:=\lVert f\rVert_{C^0([0,C_0])}$. Then there exist $\alpha>0$ and a constant $c_1$ depending only on $N$, $p$, $C_0$, $C_1$, $\diam(\Omega)$, and $\mathcal{C}_{1,r}(\Omega)$ such that
    \begin{equation}\label{C^{1,alpha}}
        |\nabla u(x) - \nabla u(y)| \le c_1|x-y|^{\alpha}, \quad \forall\, x,y\in\Omega.
    \end{equation}
\end{lemma}
\begin{proof}
    This follows directly from the results of Tolksdorf \cite{Tolksdorf1984}, Lieberman \cite{Lieberman1988}, because $\Omega$ is bounded with $C^{1,r}$ boundary, the $p$-Laplacian satisfies the required structural conditions, and $|f(u)|\le C_1$.
\end{proof}

We first record the quasiconcavity criterion of Borrelli, Mosconi, and Squassina \cite{Borrelli2024a}.

\begin{lemma}[{\cite[Theorem 1.1]{Borrelli2024a}}] \label{concave}
    Let $\Omega\subset \R^N$ be a bounded convex domain with $C^2$ boundary, let $f:\R_+ \to [0,\infty)$ be H\"{o}lder continuous with
    \[ M=\inf \{t>0:f(t)=0\},\]
    where $M=+\infty$ if $f>0$ for all $t>0$. Let $u\in W_0^{1,p}(\Omega)$ solve \eqref{Eq1}. If
   \[ x\mapsto \frac{f(x)}{x^{p-1}} \text{ is nonincreasing},\quad x\mapsto \frac{ e^{(p-1)x} }{ f(e^x) } \text{ is convex on $(-\infty,\log M)$},\]
    then $\log u$ is concave.
\end{lemma}

\begin{lemma}\label{C^{2,beta}}
    Let $s\in (0,\lVert u\rVert_{L^\infty(\Omega)}]$. Suppose $f\in C^\beta([0,s])$ and $\Omega$ is of class $C^{2,\beta}$. If $|\nabla u|>m>0$ in $\{0\le u \le s\}$, then for every $s_1<s$, $u\in C^{2,\eta}(\{0\le u\le s_1\})$ with some $\eta \le \beta$.
\end{lemma}
\begin{proof}
    Throughout this proof, $C_0:=\lVert u\rVert_{C^1(\bar\Omega)}$ and $C_1:=\lVert f\rVert_{C^0([0,C_0])}$.
    At points where $|\nabla u|\neq 0$, equation \eqref{Eq1} can be written as
    \begin{equation}\label{Eq2}
        -\sum_{i,j=1}^N \bigl(\delta_{ij}+(p-2)\frac{u_iu_j}{|\nabla u|^2}\bigr)u_{ij} = f(u) |\nabla u|^{2-p}.
    \end{equation}
    The matrix $A_{ij}(x):= \delta_{ij}+(p-2)\frac{u_iu_j}{|\nabla u|^2}$ has eigenvalues $1,p-1$ and is uniformly elliptic as long as $1<p<\infty$. By Lemma \ref{C^{1,alpha}}, $u\in C^{1,\alpha}(\Omega)\cap C^1(\bar\Omega)$ for some $\alpha\in(0,1]$ depending on $N$, $p$, $C_0$, $C_1$, $\diam(\Omega)$ and $\mathcal{C}_{2,\beta}(\Omega)$, and so $\nabla u\in C^\alpha(\bar\Omega)$. Hence
    \[ A_{ij} \in C^\alpha,\quad f(u)|\nabla u|^{2-p}\in C^\eta,\quad \text{in $\{|\nabla u|\neq 0\}$}\]
    with $\eta = \min\{\alpha,\beta\}$. Since $\Omega\in C^{2,\beta}$ and $u=0$ on $\partial\Omega$, the Schauder estimates imply that
    \[ u\in C^{2,\eta}(\{0\le u\le s_1\}),\quad \forall\, 0<s_1<s.\]
    More precisely, the $C^{2,\eta}$-norm of $u$ can be controlled by a constant depending only on $N$,$p$, $m$, $\lVert f\rVert_{C^\beta([0,s])}$, $\lVert u\rVert_{C^{1,\alpha}(\bar\Omega)}$, $\diam(\Omega)$ and $\mathcal{C}_{2,\beta}(\Omega)=(R_\Omega,M_\Omega)$:
    \[ \lVert u\rVert_{C^{2,\eta}(\{0\le u\le s_1\})} \le C(N,p,m,R_\Omega,M_\Omega) (\lVert u\rVert_{L^\infty(\Omega)}+ \lVert f(u)|\nabla u|^{2-p}\rVert_{C^{\eta}(\{0\le u < s\})}).\]
    This estimate is uniform for any $0<s_1<s$.
\end{proof}

\begin{lemma}\label{Lem convex}
    Suppose that $\Omega$ is uniformly convex and $\epsilon_0<\mathbf{c}/2$. Under the assumptions of Lemma \ref{C^{2,beta}}, let $c_1 = \lVert u\rVert_{C^{1,\alpha}(\bar\Omega)}$, $c_2(s) =\sup\limits_{0<s_1<s}\lVert u\rVert_{C^{2,\eta}(\{0\le u\le s_1\})}$, and
    \begin{equation}\label{kappa}
        \kappa_u(x) := \min_{\tau\perp \nabla u(x),\, |\tau|=1} \frac{-D^2u(x)[\tau,\tau]}{|\nabla u(x)|},
    \end{equation}
    then
    \begin{equation}\label{convex}
        \kappa_u(x) \ge \kappa_0 - C_4 d(x,\partial\Omega)^\mu,\quad \forall\, x\in\{0\le u < \min\big\{s, \frac{\mathbf{c}/2-\epsilon_0}{c_2(s)}\cdot \frac{\mathbf{c}}{2}\}\big\},
    \end{equation}
    where $\kappa_0>0$ is the minimal curvature of $\partial\Omega$, $\mu=\min\{\eta,1/2\}$ and $C_4>0$ is a constant depending only on $\kappa_0,\mathbf{c}, N, c_2(s)$.
    
    Consequently, if $\tilde{s}=\bigl(\frac{\mathbf{c}/2-\epsilon_0}{c_1}\bigr)^{1/\alpha}\frac{\mathbf{c}}{2}$ and $f\in C^\beta([0,\tilde{s}])$, then, for
    \[ s_1<\min\bigl\{\bigl(\frac{\mathbf{c}/2-\epsilon_0}{c_1}\bigr)^{1/\alpha}, \bigl(\frac{\kappa_0}{C_4}\bigr)^{1/\mu}, \frac{{\mathbf{c}/2-\epsilon_0}}{c_2(\tilde{s})}\bigr\}\frac{\mathbf{c}}{2},\] 
    the set $\{u>s_1\}$ is convex.
\end{lemma}
\begin{proof}
    Write $c_2=c_2(s)$ for simplicity. Since $u=0$ on $\partial\Omega$ and $u>0$ in $\Omega$, we have $\nabla u(x) = |\nabla u(x)|\nu(x)$ on $\partial\Omega$. Since $\partial\Omega$ is uniformly convex and of class $C^{2,\beta}$, consider a rectangular coordinate frame with origin at $x$, the $x_N$-axis being along $\nu(x)$. Then $\partial\Omega$ is locally expressed as
    \[ x_N = \phi(x_1,\dots,x_{N-1}),\quad \phi\in C^{2,\beta}.\]
    Since $u=0$ on $\partial\Omega$, we have
    \begin{equation}\label{boundary}
        u(x_1,\dots,x_{N-1},\phi)=0.
    \end{equation}
    Write $\phi_k:=\frac{\partial \phi}{\partial x_k}$ and $u_k:=\frac{\partial u}{\partial x_k}$ for simplicity. Under this coordinate, the inner normal of $\Omega$ is
    \[ (1+\sum_{k=1}^{N-1} \phi_k^2)^{-1/2}(-\phi_1,\dots,-\phi_{N-1},1).\] 
    Since $\nu(x) = (0,\dots,0,1)$, we obtain $\phi_i=0$, $i=1,\dots,N-1$, at $x$. Differentiating \eqref{boundary} with respect to $x_i$, $i=1,\dots,N-1$, we obtain that in a small neighborhood of $x$,
    \begin{equation}\label{boundary2}
        u_i + u_N\phi_i = 0.
    \end{equation}
    As a result, at $x$, we have $u_i=0$ for $i=1,\dots,N-1$ and $u_N = |\nabla u|$. Differentiating \eqref{boundary2} with respect to $x_j$, $j=1,\dots,N-1$, we obtain that at $x$,
    \begin{equation}\label{uij at x}
        u_{ij} + |\nabla u| \phi_{ij} = 0,\quad i,j=1,\dots,N-1.
    \end{equation}
    Hence, for any unit vector $\tau\perp \nabla u(x)$, writing $\tau=(\tau_1,\dots,\tau_{N-1},0)$ gives
    \begin{align}\label{curve1}
        \frac{-D^2u(x)[\tau,\tau]}{|\nabla u(x)|} = -\sum_{i,j=1}^{N-1} \frac{u_{ij}}{|\nabla u|}\tau_i\tau_j = \sum_{i,j=1}^{N-1}\phi_{ij}\tau_i\tau_j \ge \kappa_0.
    \end{align}
    Hence \eqref{convex} holds for all $x\in \partial\Omega$.

    For $x\in \{0\le u<s\}$, choose $y\in\partial\Omega$ such that $|x-y|=d(x,\partial\Omega)$. Then, by the preceding argument, $\frac{x-y}{|x-y|}=\nu(y)$ and
    \begin{equation}\label{est for xy}
        \nabla u(x) \cdot \nu(y) \ge |\nabla u(y)| - c_2|x-y|,\quad |\nabla u(x)| \le |\nabla u(y)| + c_2|x-y|.
    \end{equation}
    As a result,
    \[ \bigl|\frac{\nabla u(x)}{|\nabla u(x)|} -\nu(y)\bigr|^2 = 2 - 2\frac{\nabla u(x)\cdot\nu(y)}{|\nabla u(x)|} \le 2 - 2\frac{|\nabla u(y)|-c_2|x-y|}{|\nabla u(y)|+c_2|x-y|} \le \frac{4c_2|x-y|}{\mathbf{c}-\epsilon_0}.\]
    For $\tau\perp \nabla u(x)$ and $|\tau|=1$,
    \begin{equation}\label{est tau1}
        |\tau\cdot\nu(y)| = \bigl|\tau\cdot \bigl(\frac{\nabla u(x)}{|\nabla u(x)|}-\nu(y)\bigr)\bigr| \le \frac{2c_2^{1/2}|x-y|^{1/2}}{(\mathbf{c}-\epsilon_0)^{1/2}}.
    \end{equation}
    Set $\tau_T=\tau-(\tau\cdot\nu(y))\nu(y)$. Then $|\tau_T|^2 = |\tau|^2- |\tau\cdot\nu(y)|^2 = 1-|\tau\cdot\nu(y)|^2$. By \eqref{curve1}--\eqref{est tau1},
    \begin{align*}
        &\frac{-D^2u(x)[\tau,\tau]}{|\nabla u(x)|}\\
        &= \frac{(D^2u(y)-D^2u(x))[\tau,\tau]}{|\nabla u(x)|}+\frac{-D^2u(y)([\tau,\tau]-[\tau_T,\tau_T])}{|\nabla u(x)|} + \frac{-D^2u(y)[\tau_T,\tau_T]}{|\nabla u(x)|}\\
        &\ge -\frac{c_2|x-y|^{\eta}}{|\nabla u(x)|} - \frac{(2N-1)c_2|\tau\cdot\nu(y)|}{|\nabla u(x)|} + \frac{\kappa_0|\tau_T|^2|\nabla u(y)|}{|\nabla u(x)|} \\
        &\ge \kappa_0 \bigl(1-\frac{c_2|x-y| + |\tau\cdot\nu(y)|^2|\nabla u(y)|}{|\nabla u(y)|+c_2|x-y|}\bigr) - \frac{c_2|x-y|^\eta + (2N-1)c_2|\tau\cdot\nu(y)|}{|\nabla u(y)|-c_2|x-y|} \\
        &\ge \kappa_0 - \frac{\kappa_0\big(c_2 + \frac{4c_2}{\mathbf{c}-\epsilon_0} \big)|x-y|}{\mathbf{c}-\epsilon_0}-\frac{c_2|x-y|^{\eta}+ 2(2N-1)c_2^{3/2}(\mathbf{c}-\epsilon_0)^{-1/2}|x-y|^{1/2}}{\mathbf{c}-\epsilon_0 - c_2|x-y|} \\
        & \ge \kappa_0 - C_4|x-y|^{\min\{\eta,1/2\}}
    \end{align*}
    as long as $d(x,\partial\Omega) = |x-y|\le \frac{\mathbf{c}/2-\epsilon_0}{c_2}$ and $\epsilon_0 < \mathbf{c}/2$. Here $C_4$ is a constant depending only on $\kappa_0,c_2,\mathbf{c}, N$. Since by Corollary \ref{Cor1} (2), $d(x,\partial\Omega) \le \frac{2}{\mathbf{c}}u(x) < \frac{2s}{\mathbf{c}}$, we derive \eqref{convex}.

    If $f\in C^\beta([0,\tilde{s}])$, for $s_1<\tilde{s}=\bigl(\frac{\mathbf{c}/2-\epsilon_0}{c_1}\bigr)^{1/\alpha}\frac{\mathbf{c}}{2}$, by Corollary \ref{Cor1} (2), $|\nabla u|>\mathbf{c}/2>0$ in $\{0<u<\tilde{s}\}$, $\Omega_{2s_1/\mathbf{c}}\subset \{u>s_1\}$ and $\partial\{u>s_1\}$ is a $C^1$ hypersurface. Hence, when $s_1<\min\bigl\{\bigl(\frac{\kappa_0}{C_4}\bigr)^{1/\mu},\frac{\mathbf{c}/2 - \epsilon_0}{c_2(\tilde{s})}\bigr\}\frac{\mathbf{c}}{2}$, by $2s_1/\mathbf{c} < \bigl(\frac{\kappa_0}{C_4}\bigr)^{1/\mu}$ and \eqref{convex}, we have
    \[ \kappa_u(x) \ge \kappa_0 - C_4 d(x,\partial\Omega)^\mu \ge \kappa_0 - C_4(2s_1/\mathbf{c})^\mu > 0,\quad \forall\, x\in \partial\{u>s_1\},\]
    and hence the second fundamental form of $\partial\{u>s_1\}$ is positive definite at every point.
    
    Since $|\nabla u|>\mathbf{c}/2>0$ in $\{0<u<\tilde{s}\}$, define $X(x):= \frac{\nabla u(x)}{|\nabla u(x)|^2}$ and denote its flow by $\Phi_t(x)$. Then
    \[ \frac{d}{dt} u(\Phi_t(x)) = \nabla u(\Phi_t(x))\cdot X(\Phi_t(x)) = 1.\]
    Thus, starting from $x\in\partial\Omega$, we have $u(\Phi_t(x))=t$. Conversely, for each $y\in \{u=s_1\}$, flowing backward yields $u(\Phi_{-s_1}(y))=0$. Therefore, $\Phi_{s_1}:\partial\Omega \to \{u=s_1\}$ is a diffeomorphism, so $\{u=s_1\}$ is connected. Hence, by the Hadamard global convexity theorem \cite{Hadamard1898, Mohammadpouri2012}, $\{u>s_1\}$ is convex.
\end{proof}

\begin{proof}[Proof of Corollary \ref{maincor1}]
    Throughout this proof, $C_0:=\lVert u\rVert_{C^1(\bar\Omega)}$ and $C_1:=\lVert f\rVert_{C^0([0,C_0])}$.

    \noindent \textbf{Case 1: $\Omega$ and $f$ satisfy (i).}

    By Lemma \ref{regular1}, condition (b) is satisfied for $L=c_1$ and some $\alpha\in (0,1]$ depending only on $N$, $p$, $\diam(\Omega)$, $C_0$, $C_1$ and $\mathcal{C}_{2,\beta}(\Omega)$.
    
    Set
    \[\tilde{s}:= \bigl(\frac{\mathbf{c}}{4c_1}\bigr)^{1/\alpha}\frac{\mathbf{c}}{2}.\]
    Since $f\in C^{\beta}([0,\tilde{s}])$ and $\epsilon_0<\mathbf{c}/4$, Lemma \ref{C^{2,beta}} gives $u\in C^{2,\eta}(\{0\le u \le s_1\}$ for every $0<s_1<\tilde{s}$.
    
    By Lemma \ref{Lem convex}, for
    \[ 0\le c\le s_*:=\min\bigl\{\bigl(\frac{\mathbf{c}}{4c_1}\bigr)^{1/\alpha}, \bigl(\frac{\kappa_0}{C_4}\bigr)^{1/\mu}, \frac{\mathbf{c}}{4c_2(\tilde{s})}\bigr\}\frac{\mathbf{c}}{2},\]
    the set $\{u>c\}$ is convex, and
    \[ u \in C^{2,\eta}(\{0\le u\le s_1\}),\]
    for any $0<s_1<s_*$ with $\eta=\min\{1/2,\beta,\alpha\}$. Then conditions (c) and (d) hold.
    
    Since now we have $u\in C^2(\{0\le u \le s_1\})$ for any $0<s_1<s_*$, we can replace all the $(L,\alpha)$ in the proof of Theorem \ref{thm1} and \ref{mainthm} by $(c_2(\tilde{s}), 1)$ and get \eqref{main}:
    \[ |\Omega\Delta B_{\Omega}(q)| \le C''\epsilon_0^{1/4} \]
    for $C'' = 2N(1+2(2/\mathbf{c})^{p-1} C_1 \diam(\Omega)) C_9$. The constant $C_9$ is
    \[ C_9 = C_8\bigl(\frac{2\diam(\Omega)}{\rho_i}\bigr)^2\Big( A^{-1/2} + A^{1/2} \Big) + 2(2/\mathbf{c})^{p}C_2A^{1/2}\]
    with
    \begin{align*}
        & C_2 = C_1|\Omega|,\\
        & C_3 = 4pC_0C_1|\Omega|\diam(\Omega) \big( 2C_1(2/\mathbf{c})^p + (p-1)2^{2p}\mathbf{c}^{-1} (c_1+1) \big),\\
        & C_4 \,\,\text{as in Lemma \ref{Lem convex}},\\
        & C_5 = 2(p-1) \max\{(1/2)^{p-2}, (3/2)^{p-2}\} ,\\
        & C_6 = \max\bigl\{\frac{C_3}{\mathbf{c}^{p-1}}, \frac{C_5C_2}{\mathbf{c}}\bigl(\frac{2}{\mathbf{c}}\bigr)^{p-1}, \frac{C_5C_2c_2(\tilde{s})}{\mathbf{c}}\bigl(\frac{2}{\mathbf{c}}\bigr)^{p}\bigr\} \\
        & C_7 \,\,\text{as in \eqref{steiner stab}},\\
        & C_8 = 2C_6^{1/2} C_7|\Omega|^{\frac{N+1}{2N}}(Nw_N^{1/N})^{-1/2}, \\
        & A = \min\bigl\{\bigl(\frac{\kappa_0}{C_4}\bigr)^{1/\mu}, \frac{\mathbf{c}}{4c_2(\tilde{s})}, \bigl(\frac{\mathbf{c}}{4c_1}\bigr)^{1/\alpha}, \frac{\rho_i}{2}, \frac{2}{\mathbf{c}}\bigr\}\frac{\mathbf{c}}{2}.
    \end{align*}
    We deduce that the constant $C''$ depends only on $N$, $p$, $\mathcal{C}_{2,\beta}(\Omega)$, $\mathbf{c}$, $\lVert f\rVert_{C^\beta[0,C_0]}$, $\lVert u\rVert_{C^1(\bar\Omega)}$, $\kappa_0$, $\diam(\Omega)$ and the inradius $\rho_i$ of $\Omega$.

    \noindent \textbf{Case 2: $\Omega$ and $f$ satisfy (ii).}

    By Lemma \ref{regular1}, condition (b) is satisfied for $L=c_1$ and some $\alpha\in (0,1]$ depending only on $N$, $p$, $\diam(\Omega)$, $\mathcal{C}_2(\Omega)$, $C_0$, $C_1$.
    
    By Lemma \ref{concave}, $\log u$ is concave. Hence condition (c) holds for every $r\in [0,C_0]$, and we may take $s_*=C_0$. Condition (d) is satisfied because $\Omega$ is convex and $\partial\Omega$ is of class $C^2$.

    Hence, by Theorem \ref{mainthm}, we see that \eqref{main est 2} holds for some $C''$ depending only on $N$, $p$, $\mathbf{c}$, $\lVert f\rVert_{C^\beta[0,C_0]}$, $\lVert u\rVert_{C^1(\bar\Omega)}$, $\diam(\Omega)$, $\mathcal{C}_2(\Omega)$ and the inradius $\rho_i$ of $\Omega$.

    If $\partial\Omega$ is of class $C^{2,\beta}$, set
    \[\tilde{s}:= \bigl(\frac{\mathbf{c}}{4c_1}\bigr)^{1/\alpha}\frac{\mathbf{c}}{2}\]
    The H\"{o}lder continuity of $f$ and Lemma \ref{C^{2,beta}} imply that $u\in C^{2,\eta}(\{0\le u \le s_1\})$ for any $0<s_1<\tilde{s}$, where $\eta = \min\{\beta,\alpha\}$. Thus, we may replace $(L,\alpha)$ by $(c_2(\tilde{s}),1)$ in the proofs of Theorems \ref{thm1} and \ref{mainthm}. Therefore, one may take $\alpha=1$ in \eqref{main est 2}.
\end{proof}

\subsection{General Divergence-Form Elliptic Operators}
\begin{proof}[Proof of Corollary \ref{Cor2}]
    Set
    \[ C_1=\lVert f\rVert_{C^\beta(\R)},\quad  C_0 = \lVert u\rVert_{C^{1}(\bar\Omega)}, \quad g_0 = \bigl\lVert \frac{xg'(x)}{g(x)}\bigr\rVert_{C^{0,\theta}[\mathbf{c}/2,{3\mathbf{c}}/2]},\]  
    and define $G(x) :=\int_0^x g(y)\,dy$. The assumptions on $g$ imply that $G$ is strictly convex and increasing. After a rotation, we may assume that $e_N \in A_u$.

    \noindent \textbf{1. $C^{1,\alpha}$ regularity and its consequences.}

    For the generalized equation, the structural conditions in \cite{Lieberman1988} imply the analogue of Lemma \ref{regular1}. Let $c_1 := \lVert u\rVert_{C^{1,\alpha}(\bar\Omega)}$. Then conclusions (1) and (2) of Corollary \ref{Cor1} hold with $L=c_1$. Integrating \eqref{g ver eq} yields the following analogue of conclusion (3):
    \[ \int_{\{u=r\}} g(|\nabla u|)\,d\mathcal{H}^{N-1}(x) \le C_2 = C_1|\Omega|.\]

    \noindent \textbf{2. $C^{2,\eta}$ regularity and convexity.}
    
    Since $\epsilon_0<\min\{1,\mathbf{c}/4\}$, set
    \[\tilde{s}:=\bigl(\frac{\mathbf{c}}{4c_1}\bigr)^{1/\alpha}\frac{\mathbf{c}}{2}.\]
    When $x\in \{0<u<\tilde{s}\}$, $|\nabla u|\in [\mathbf{c}/2,3\mathbf{c}/2]$ and \eqref{g ver eq} is equivalent to
    \[  - \sum_{i,j=1}^N \bigl( \frac{g'(|\nabla u|)|\nabla u| - g(|\nabla u|)}{g(|\nabla u|)} \frac{u_i u_j}{|\nabla u|^2} + \delta_{ij} \bigr) u_{ij} = f(u)\frac{|\nabla u|}{g(|\nabla u|)}. \]
    The coefficient matrix on the left-hand side has eigenvalues $1$ and $\frac{g'(|\nabla u|)|\nabla u|}{g(|\nabla u|)}\in [\lambda_g,\Lambda_g]$. Thus, it is uniformly elliptic and belongs to $C^{0,\min\{\alpha,\theta\}}$. Since $\nabla u \in C^{\alpha}(\bar\Omega)$, $f\in C^{\beta}(\R)$, and $\partial \Omega \in C^{2,\beta}$, the Schauder estimates give that for any $0<s_1<\tilde{s}$,
    \begin{align*}
        & \lVert u\rVert_{C^{2,\eta}(\{0\le u\le s_1\})} \\
        & \le C(\mathcal{C}_{2,\beta}(\Omega),\diam(\Omega), N,p,\lambda_g,\Lambda_g) (\lVert u\rVert_{L^\infty(\Omega)} + \lVert f(u)|\nabla u|/ g(|\nabla u|) \rVert_{C^{\eta}(\Omega)})=:c_2
    \end{align*}
    for $\eta=\min\{\beta,\alpha,\theta\}$. Here $c_2$ depends only on $N$, $\lambda_g$, $\Lambda_g$, $\mathbf{c}$, $\diam(\Omega)$, $\mathcal{C}_{2,\beta}(\Omega)$, $\lVert f\rVert_{C^{\beta}(\R)}$, $\lVert u\rVert_{C^1(\bar\Omega)}$ and $g_0$.
    
    Repeating the argument of Lemma \ref{Lem convex}, we find that $\{u>s_1\}$ is convex for all
    \[0\le s_1 < s_*:=\min\bigl\{\bigl(\frac{\mathbf{c}}{4c_1}\bigr)^{1/\alpha}, \bigl(\frac{\kappa_0}{C_4}\bigr)^{1/\mu}, \frac{\mathbf{c}}{4c_2}\bigr\}\frac{\mathbf{c}}{2}.\]
    Since $G$ is strictly convex and increasing, for $\lambda_1,\lambda_2 \in (0,1/2)$ and $\xi\in(0,1)$,
    \begin{align*}
        &\int_K \xi b(\lambda_1)G\bigl(\frac{a(\lambda_1)}{b(\lambda_1)}\bigr) + (1-\xi) b(\lambda_2)G\bigl(\frac{a(\lambda_2)}{b(\lambda_2)}\bigr) \\
        & \ge \int_K \big( \xi b(\lambda_1) + ( 1-\xi) b(\lambda_2)\big)G\bigl(\frac{\xi a(\lambda_1) + (1-\xi) a(\lambda_2 )}{\xi b(\lambda_1) + (1-\xi) b(\lambda_2)}\bigr) \\
        & = \int_K b(\xi \lambda_1 + (1-\xi)\lambda_2) G\bigl(\frac{\xi a(\lambda_1) + (1-\xi) a(\lambda_2 )}{b(\xi \lambda_1 + (1-\xi)\lambda_2)}\bigr) \\
        & \ge \int_K b(\xi \lambda_1 + (1-\xi)\lambda_2) G\bigl(\frac{a(\xi \lambda_1 + (1-\xi)\lambda_2 )}{b(\xi \lambda_1 + (1-\xi)\lambda_2)}\bigr),
    \end{align*}
    where $b(\lambda), a(\lambda)$ and $K$ are defined as in Lemma \ref{good convex}. Hence, by $e_N\in A_u\subset A_{u_{\gamma,s}}$,
    \[ \int_{\R^N} G(|\nabla u_{\gamma,s}^t|)\,dx = \int_K b(\lambda) G\bigl(\frac{a(\lambda)}{b(\lambda)}\bigr) + b(1-\lambda) G\bigl(\frac{a(1-\lambda)}{b(1-\lambda)}\bigr)\]
    is convex with respect to $\lambda = \frac{1-e^{-t}}{2} \in [0,1/2]$ as long as $0<s< s_*$.
    
    \noindent \textbf{3. Difference of the energy.}

    We claim that if $0<s<s_*$ and $0<\gamma<\min\{\epsilon_0,s/2\}$, then
    \begin{equation}\label{g ver Steiner est}
        \int_{\R^N} G(|\nabla u_{\gamma,s}|) - \int_{\R^N}G(|\nabla (u_{\gamma,s})^\infty|) \le \tilde{C_3}\epsilon_0
    \end{equation}
    with $\tilde{C_3}$ depending only on $N$, $\mathbf{c}$, $g$, $C_0$, $\lVert f\rVert_{C^\beta(\R)}$, $\diam(\Omega)$ and $\mathcal{C}_{2,\beta}(\Omega)$.

    In fact, we only need to modify the proof of Proposition \ref{Prop 1} slightly. By the definition of $C_0$, the Lipschitz constant of $u$ is at most $C_0$. In \textbf{Step one}, when testing \eqref{g ver eq} by $(u-\gamma)_+^t - (u-\gamma)_+$, we derive that
    \begin{align*}
        \int_{\Omega} f(u) [(u-\gamma)_+^t - (u-\gamma)_+]\,dx ={}& \int_{\Omega} g(|\nabla u|)\frac{\nabla u}{|\nabla u|}\cdot \nabla ((u-\gamma)_+^t - (u-\gamma)_+)\,dx \\
        \le {}& \int_{\{u^t>\gamma\}}g(|\nabla u|)|\nabla u^t|\,dx - \int_{\{u>\gamma\}} g(|\nabla u|)|\nabla u|\,dx,
    \end{align*}
    and so
    \begin{align*}
        & \int_{\{u^t>\gamma\}}g(|\nabla u|)|\nabla u^t|\,dx - \int_{\{u>\gamma\}} g(|\nabla u|)|\nabla u|\,dx \\
        & \ge -2C_1C_0\diam(\Omega)t|\{0<u<\gamma\}|+o(t),
    \end{align*}
    where $\lVert f(u)\rVert_{L^\infty(\Omega)} \le C_1$ by the definition of $C_1$.

    In \textbf{Step two}, let $h(y)=yg(y)-G(y)$, then $h'(y) = yg'(y)$. Since $u\in C^{2,\eta}(\{0<u<s_*\})$, when
    \[\gamma < \frac{\mathbf{c}^2}{8c_2}-C_0\diam(\Omega)t,\]
    the gradient estimate becomes
    \[ ||\nabla u|-\mathbf{c}| \le c_2 d(x,\partial\Omega) + \epsilon_0 \le c_2\frac{2}{\mathbf{c}} (C_0\diam(\Omega)t+\gamma) + \epsilon_0 < \mathbf{c}\]
    for every $|u-\gamma|<C_0\diam(\Omega)t$. Thus,
    \begin{align*}
        & \int_{\{u^t>\gamma\}} h(|\nabla u|)\,dx - \int_{\{u>\gamma\}} h(|\nabla u|)\, dx\\
        ={}& \int_{\{u^t>\gamma\}} \big(h(|\nabla u|)-h(\mathbf{c})\big)\,dx - \int_{\{u>\gamma\}} \big(h(|\nabla u|)-h(\mathbf{c})\big)\, dx \\
        \le{}& \int_{\{\gamma-C_0\diam(\Omega)t\le u \le \gamma + C_0\diam(\Omega)t\}} |h(|\nabla u|)-h(\mathbf{c})|\,dx \\
        \le{}& \sup_{x\in [0, 2\mathbf{c}]} |h'(x)| \int_{\gamma-C_0\diam(\Omega)t}^{\gamma + C_0\diam(\Omega)t} \int_{\{u=y\}} |\nabla u|^{-1}||\nabla u|-\mathbf{c}|\,d\mathcal{H}^{N-1}dy \\
        \le{}& \sup_{x\in [0, 2\mathbf{c}]} |xg'(x)|\Big( \frac{2c_2}{\mathbf{c}}(C_0\diam(\Omega)t+\gamma)+\epsilon_0\Big)\frac{2/\mathbf{c}}{g(\mathbf{c}/2)}\\
        & \qquad \cdot \int_{\gamma-C_0\diam(\Omega)t}^{\gamma + C_0\diam(\Omega)t} \int_{\{u=y\}} g(|\nabla u|)\,d\mathcal{H}^{N-1}dy \\
        \le{}& \Lambda_g g(2\mathbf{c})\Big(  \frac{2c_2}{\mathbf{c}}(C_0\diam(\Omega)t+\gamma)+\epsilon_0\Big)\frac{2/\mathbf{c}}{g(\mathbf{c}/2)}2C_0C_2\diam(\Omega)t.
    \end{align*}
    Indeed, \eqref{eigen} gives $\sup_{x\in [0, 2\mathbf{c}]} |xg'(x)| \le \sup_{x\in [0,2\mathbf{c}]} \Lambda_g g(x) \le \Lambda_g g(2\mathbf{c})$.

    In \textbf{Step three}, \eqref{Lem3.2 in Wei} gives, for
    \[C_0\diam(\Omega)t<\gamma < \frac{\mathbf{c}^2}{8c_2}-C_0\diam(\Omega)t,\]
    there holds
    \begin{align*}
        &\int_{\{u^t>\gamma\} } G(|\nabla u^t|)\,dx - \int_{ \{u>\gamma\} }G(|\nabla u|)\,dx \\
        &\ge \int_{\{u^t>\gamma\}} g(|\nabla u|)|\nabla u^t|\,dx - \int_{\{u>\gamma\}} g(|\nabla u|)|\nabla u|\,dx\\
        & \quad - \bigl(\int_{\{u^t>\gamma\}} h(|\nabla u|)\,dx - \int_{\{u>\gamma\}} h(|\nabla u|)\,dx\bigr)\\
        &\ge \Big(-2C_1C_0\diam(\Omega)|\{0<u<\gamma\}| \\
        & \qquad - \sup_{x\in [0, 2\mathbf{c}]} |xg'(x)|\Big(\frac{2c_2}{\mathbf{c}}(C_0\diam(\Omega)t+\gamma)+\epsilon_0\Big)\frac{2/\mathbf{c}}{g(\mathbf{c}/2)}2C_0C_2\diam(\Omega)\Big)t + o(t).
    \end{align*}
    Thus, by \textbf{Step four}, since
    \[ |\{0<u<\gamma\}| \le C_2 \frac{2/\mathbf{c}}{g(\mathbf{c}/2)} \gamma,\]
    we can deduce \eqref{g ver Steiner est} with $\tilde{C_3} =  \frac{4/\mathbf{c}}{g(\mathbf{c}/2)}C_0C_1|\Omega|\diam(\Omega) \big( \Lambda_g g(2\mathbf{c})(2c_2/\mathbf{c} + 1) + C_1\big)$.

    \noindent \textbf{4. Difference of the perimeter.}

    The Lipschitz property still holds: $\lVert \nabla u_{\gamma,s}^\infty\rVert_{L^\infty(\Omega)}\le \lVert \nabla u_{\gamma,s}\rVert_{L^\infty(\Omega)} \le \mathbf{c}+\epsilon_0 + c_2 (2s/\mathbf{c})$. By the definition of $s_*$ and \eqref{c quarter prop},
    \[c_2(2s/\mathbf{c})<\mathbf{c}/4<\mathbf{c}/2-\epsilon_0.\]
    
    Note that $\int_{\R^N} G(|\nabla u_{\gamma,s}|)\,dx = \int_\gamma^s \int_{\{u=y\}} \frac{G(|\nabla u|)}{|\nabla u|}\,d\mathcal{H}^{N-1}\,dy$. Let $H(x) := \frac{G(x)}{x}$. Since $g$ is increasing, we have $G(x) = \int_0^x g(y)\,dy \le xg(x)$, and hence $H'(x) = \frac{xg(x)-G(x)}{x^2} \ge 0$. Thus, $H$ is increasing.
    \[ \int_\gamma^s \int_{\{u=y\}} \frac{G(|\nabla u|)}{|\nabla u|}\,d\mathcal{H}^{N-1}\,dy \ge H(\mathbf{c}-\epsilon_0 - c_2 (2s/\mathbf{c}) )\int_\gamma^s P(E_y)\,dy.\]
    Similarly,
    \[ \int_\gamma^s \int_{\{u^\infty=y\}} \frac{G(|\nabla u^\infty|)}{|\nabla u^\infty|}\,d\mathcal{H}^{N-1}\,dy \le H(\mathbf{c}+\epsilon_0 + c_2 (2s/\mathbf{c}) )\int_\gamma^s P(E_y^\infty)\,dy.\]
    Thus,
    \begin{align*}
        \tilde{C_3}\epsilon_0 \ge{}& \int_{\R^N} G(|\nabla u_{\gamma,s}|)\,dx-\int_{\R^N} G(|\nabla u_{\gamma,s}^\infty|)\,dx \\
        \ge{}& H(\mathbf{c}+\epsilon_0 + c_2 (2s/\mathbf{c}) ) \int_\gamma^s (P(E_y)-P(E_y^\infty)) \,dy\\
        &\quad - \frac{2 g(\mathbf{c}+\epsilon_0 + c_2 (2s/\mathbf{c}) )}{\mathbf{c}-\epsilon_0 - c_2 (2s/\mathbf{c}) }(\epsilon_0 + c_2 (2s/\mathbf{c}) )\int_\gamma^s P(E_y)\,dy\\
        \ge{}& H(\mathbf{c})\int_\gamma^s (P(E_y)-P(E_y^\infty))\,dy - \frac{4 g(3\mathbf{c}/2)}{\mathbf{c}}(\epsilon_0 + c_2 (2s/\mathbf{c}) )\int_\gamma^s P(E_y)\,dy,
    \end{align*}
    where the second and third inequalities follow because
    \[ H'(x) = \frac{g(x)}{x}-\frac{G(x)}{x^2} \le \frac{g(x)}{x} \le \frac{g(\mathbf{c}+\epsilon_0 + c_2 (2s/\mathbf{c}) )}{\mathbf{c}-\epsilon_0 - c_2 (2s/\mathbf{c})} \le \frac{g(3\mathbf{c}/2)}{\mathbf{c}/2},\]
    for all $x\in [\mathbf{c}-\epsilon_0 - c_2 (2s/\mathbf{c}) ,\mathbf{c}+\epsilon_0 + c_2 (2s/\mathbf{c}) ]\subset [\mathbf{c}/2,3\mathbf{c}/2]$. By $|\nabla u|>\mathbf{c}/2$ in $\{0\le u\le s_*\}$ and $g$ is monotone,
    \[ P(E_y) = \int_{\{u=y\}} 1\,d\mathcal{H}^{N-1} \le \frac{1}{g(\mathbf{c}/2)}\int_{\{u=y\}} g(|\nabla u|)\,d\mathcal{H}^{N-1} \le \frac{C_2}{g(\mathbf{c}/2)},\quad \forall\, \gamma\le y \le s.\]
    Thus, we deduce that
    \[ \int_\gamma^s \big(P(E_y)-P(E_y^\infty)\big)\,dy \le \frac{\tilde{C_3}}{H(\mathbf{c})}\epsilon_0 + \frac{4C_2g(3\mathbf{c}/2)}{\mathbf{c} g(\mathbf{c}/2)}( \epsilon_0 + c_2 (2s/\mathbf{c}) )s.\]
    By $\gamma<s/2$, there exists $s_0\in [s/2,s]$ such that
    \begin{align*}
        P(E_{s_0})-P(E_{s_0}^\infty) \le{}& \frac{2\tilde{C_3}}{H(\mathbf{c})}\frac{\epsilon_0}{s} + \frac{8C_2g(3\mathbf{c}/2)}{\mathbf{c} g(\mathbf{c}/2)}( \epsilon_0 + c_2 (2s/\mathbf{c}) )\\
        \le{}& 2\tilde{C_6} (\epsilon_0 + s + \epsilon_0/s),
    \end{align*}
    where
    \[ \tilde{C_6} = \max\bigl\{\frac{\tilde{C_3}}{H(\mathbf{c})}, \frac{4C_2g(3\mathbf{c}/2)}{\mathbf{c} g(\mathbf{c}/2)}, \frac{4C_2g(3\mathbf{c}/2)}{\mathbf{c} g(\mathbf{c}/2)}\frac{2c_2}{\mathbf{c}}\bigr\}. \]

    \noindent \textbf{5. Stability.}
    
    If $s<1$, repeating the argument of Theorem \ref{thm1} yields
    \[ \inf_{y\in\R} |\Omega \Delta (\Omega^\infty+(0,y))| \le \tilde{C_8} \mathcal{E}^2(s^{1/2} + \epsilon_0^{1/2}/s^{1/2}) + \frac{2C_2s}{g(\mathbf{c}/2)(\mathbf{c}/2)},\]
    where
    \[ \tilde{C_8} = 2\tilde{C_6}^{1/2} C_7 |\Omega|^{\frac{N+1}{2N}}(N w_N^{1/N})^{-1/2}\] 
    and $\mathcal{E}$ is the ratio of the circumradius to the inradius of $E_{s_0}$. Let $s=\min\{s_*,\frac{\mathbf{c}\rho_i}{4},1\}\epsilon_0^{1/2}$, where $\rho_i$ is the inradius of $\Omega$. Then
    \[ \mathcal{E} \le \frac{\diam(\Omega)}{\rho_i - 2s_0/\mathbf{c}} \le \frac{2\diam(\Omega)}{\rho_i},\]
    and so
    \[ \inf_{y\in\R} |\Omega \Delta (\Omega^\infty+(0,y))| \le \tilde{C_9}\epsilon_0^{1/4}\]
    with
    \[ \tilde{C_9} = \tilde{C_8}\bigl(\frac{2\diam(\Omega)}{\rho_i}\bigr)^2 (B^{1/2}+B^{-1/2}) + \frac{2C_1|\Omega|B}{g(\mathbf{c}/2)(\mathbf{c}/2)},\quad B = \min\bigl\{s_*,\frac{\mathbf{c}\rho_i}{4}, 1\bigr\}.\]
    Finally, the argument used to prove Theorem \ref{mainthm} yields
    \[ |\Omega \Delta B_\Omega(q)| \le 2N \big(1+ \frac{2}{g(\mathbf{c}/2)}C_1\diam(\Omega)\big)\tilde{C_9} \epsilon_0^{1/4}.\]
    The constant depends only on $N$, $\mathbf{c}$, $\beta$, $\kappa_0$, $g$, $\diam(\Omega)$, $\rho_i$, $\lVert f\rVert_{C^{\beta}(\R)}$, $\lVert u\rVert_{C^1(\bar\Omega)}$ and $\mathcal{C}_{2,\beta}(\Omega)$, and we are done.
\end{proof}

\section*{Acknowledgements}
The authors gratefully acknowledge the support described in the funding
statement below.

\paragraph{Funding.}
Wenming Zou is supported by the National Key Research and Development Program
of China (Grant No.~2023YFA1010001) and the National Natural Science Foundation
of China (Grant No.~12571123).

\section*{Declaration of AI use}
All theorems and their proofs in this article were completed by the authors.
The authors used GPT-5.6 for language polishing and spell-checking.

\bibliographystyle{plain}
\bibliography{ref}

\end{document}